\documentclass[11pt]{article}
\usepackage{a4wide}  
\usepackage{amsmath,amssymb,bbm}
\usepackage{mathrsfs}
\usepackage{mathenv}
\usepackage{fancybox}
\usepackage{version}
\usepackage[backend=bibtex,style=numeric,maxbibnames=99]{biblatex}
\usepackage{authblk}
\usepackage{empheq}

\usepackage{blindtext}
\usepackage{multicol}
\usepackage{color}
\usepackage{comment}
\usepackage{wrapfig}
\usepackage{graphicx}

\usepackage{hyperref}
\hypersetup{colorlinks=true, linkcolor=blue, filecolor=magenta, urlcolor=blue}
\usepackage{amsthm,cleveref}
\providecommand{\keywords}[1]{\textbf{\textit{Keywords---}} #1}

\newcommand{\R}{\mathbb{R}}
\newcommand{\NN}{\mathbb{N}}
\newcommand{\Z}{\mathbb{Z}}

\newcommand{\norme}[1]{\left\Vert #1\right\Vert}
\newcommand{\abs}[1]{\left\lvert #1\right\rvert}
\newcommand{\intervalleoo}[2]{\mathopen{]}#1\, ,#2\mathclose{[}}
\newcommand{\intervalleff}[2]{\mathopen{[}#1\, ,#2\mathclose{]}}
\newcommand{\intervalleof}[2]{\mathopen{]}#1\, ,#2\mathclose{]}}
\newcommand{\intervallefo}[2]{\mathopen{[}#1\, ,#2\mathclose{[}}

\newtheorem{theorem}{THEOREM}[section]

\newtheorem{lemma}[theorem]{LEMMA}

\newtheorem{remark}[theorem]{REMARK}

\newtheorem{proposition}[theorem]{PROPOSITION}

\begin{document}


\title{Propagation of velocity moments and uniqueness for the magnetized Vlasov--Poisson system}

\author{Alexandre Rege\footnote{ETH Z\"urich, Department of Mathematics, R\"amistrasse 101, 8092 Z\"urich Switzerland.\\
Email: \href{mailto:alexandre.rege@math.ethz.ch}{alexandre.rege@math.ethz.ch}}}

\maketitle

\begin{abstract}
We present two results regarding the three-dimensional Vlasov--Poisson
system in the full space with an external magnetic field.  First, we investigate the propagation of velocity moments for solutions to the system when the magnetic field is uniform and time-dependent. We combine the classical moment approach with an induction procedure depending on the cyclotron period $T_c=\norme{B}_{\infty}^{-1}$. This allows us to obtain, like in the unmagnetized case, the propagation of velocity moments of order $k>2$ in the full space case and of order $k>3$ in the periodic case. Second, this time taking a general magnetic field that depends on both time and position, we manage to extend a result by Miot \cite{M16} regarding uniqueness for Vlasov--Poisson to the magnetized framework.
\end{abstract}

\section{Introduction}
We begin by stating that the results in this paper were announced in the author's PhD thesis \cite{rege21}. 

We study the Cauchy problem for the three-dimensional Vlasov--Poisson system with an external magnetic field. This system is usually called the magnetized Vlasov--Poisson system, and is given by the following set of equations:

\begin{equation}\label{sys:VPwB}\tag{VPB}
\left\{
\begin{aligned}
& \partial_t f + v\cdot \nabla_x f + \left(E+v \wedge B\right) \cdot \nabla_{v}	 f= 0, \\
& f(0,x,v)=f^{in}(x,v)\geq 0.
\end{aligned}
\right.
\end{equation}
where $f^{in}$ is a positive measurable function and $f:=f(t,x,v)$ is the distribution function of particles at time $t \in \R_+$, position $x\in \R^3$ and velocity $v \in \R^3$.  The self-consistent electric field $E:=E(t,x)$ is given by:
\begin{equation}\label{def:E}
 E=-\nabla_x \mathcal{G}_3 \ast \rho,
\end{equation}
with $\mathcal{G}_3=\frac{1}{4\pi \abs{x}}$ the Green function for the Laplacian and $\rho(t,x):=\int_{\R^3} f(t,x,v)dv$ the charge particle density. The external magnetic field $B:=B(t,x)$ will be locally bounded in time and Lipschitz in position. As mentioned above we will assume that $B$ is uniform when we study the propagation of velocity of moments for solutions to \eqref{sys:VPwB}.
This system models the evolution of a set of charged particles subject to an external magnetic field $B$ that interact through the Coulomb force, and thus it is relevant for the study of various physical systems, most notably plasmas.

The mathematical theory for the unmagnetized Vlasov--Poisson system has been studied and developed in a great number of different works. In the three-dimensional framework, Arsenev \cite{AR75} was the first to prove the existence of global weak solutions through a regularization procedure that preserves the main a priori estimates. The existence of global classical solutions for general initial data was established at the beginning of the nineties in two separate works by Pfaffelmoser \cite{P92} and Lions, Perthame \cite{LP91}. The first approach, extended and developed in \cite{CC19,H93,S91}, relies on a very sharp study of the characteristics of the Vlasov--Poisson system while the latter approach, extended and developed in \cite{C99,GJP00,P14,P96,S09}, is based on estimating velocity moments of weak solutions in a very fine way by using the a priori bounds verified by the system. Even if these two approaches differ greatly, in both cases a key condition is to limit the influence of high velocities on the dynamics, by either considering initial data with compact support \cite{P92} or having a finite velocity moment of sufficiently high order ($k>3$) \cite{LP91}. More recently, Pallard combined the two approaches in \cite{P12}, where he showed how to exploit the first approach by Pfaffelmoser to prove propagation of velocity moments, extending the main result of \cite{LP91} by showing that this propagation property is also true for moments of order $2< k \leq 3$. For more information on well-posedness for Vlasov type systems, we refer the reader to the recent review \cite{GI21three}.

Going back to the magnetized Vlasov--Poisson system \eqref{sys:VPwB}, existence of weak solutions to this system can be seen as a corollary of the work by DiPerna and Lions \cite{DLvlasovmax} on the existence of renormalized solutions for Vlasov--Maxwell. The system was then studied by Golse and Saint-Raymond in the strong magnetic field limit \cite{GS99,GS03}. More recently, in a linear setting Bedrossian and Wang \cite{BW20} and Charles, Despr\'es, Weder and the author \cite{CDRW21} proved in two different ways that the external magnetic field in \eqref{sys:VPwB} destroys Landau damping. In another recent paper \cite{R21}, the author showed propagation of velocity moments for \eqref{sys:VPwB} with constant magnetic field $B=(0,0,\omega)$ (with $\omega>0$ the cyclotron frequency) by extending the method of propagation of velocity moments from \cite{LP91}. In order to extend the moment method, an important point in \cite{R21} was to establish a representation formula for the charge density $\rho$. This was carried out by explicitly computing the characteristics of the transport equation
\begin{equation}\label{eq:Vlasovlin}
\partial_t f+v\cdot \nabla_x f+ v\wedge B \cdot \nabla_v f=0
\end{equation} 
and then using the Duhamel formula, which just meant considering the Vlasov equation as the transport equation \eqref{eq:Vlasovlin} with a source given by the nonlinear term $-E\cdot \nabla_v f$. In this analysis, singularities at times $t=0,2\pi/\omega,4\pi/\omega,...$, which are just multiples of the cyclotron period $T_\omega=2\pi/\omega$, appeared in the velocity moment estimates because of the added magnetic field. This was remedied by the fact that all the estimates depended only on quantities conserved for all time and on the initial velocity moment, allowing for an induction argument to prove propagation of moments for all time. Unfortunately, in our configuration with a general magnetic field, this analysis breaks down at the first hurdle, simply because we can't explicitly compute the characteristics of \eqref{eq:Vlasovlin}, even with a smooth $B$.

The first main result of this paper, which is the continuation of \cite{R21}, is to generalize Pallard's method \cite{P12} to the magnetized Vlasov--Poisson system in the case of a uniform time-dependent magnetic field. We manage to obtain this result by combining Pallard's method with an induction argument using the cyclotron period similar to the one in \cite{R21}. However, in this paper we don't obtain explicit singularities like in \cite{R21} because, as said above, we don't compute the characteristics of the system explicitly but instead we write estimates on the evolution of the characteristics like in \cite{P12}. Indeed, with the added magnetic field, we observe that we need to work on a small time scale compared to the cyclotron period $T_c=\norme{B}_{L^\infty(\intervalleff{0}{T})}^{-1}$ to obtain estimates analogous to those in \cite{P12}. Very broadly speaking, this is due to the fact that on time scales comparable to $T_c$ or greater than $T_c$, the variation of the velocity characteristic is large and so the method in \cite{P12} fails without this assumption. This justifies the use of the induction argument and allows us to obtain the same optimal results for \eqref{sys:VPwB} with $B:=B(t)$ as in the unmagnetized case. More precisely, we obtain the propagation of velocity moments of order $k>2$ in the full space case and of order $k>3$ in the periodic case.

Let us also recall that propagation of velocity moments also implies propagation of the regularity of the initial data which means we have existence of classical solutions to \eqref{sys:VPwB}. This result is detailed in \cite[theorem 2.5]{R21} for a constant magnetic field under additional conditions on $f^{in}$ ($f^{in}$ decays faster in velocity) but can be easily extended to the case of a general magnetic field $B:=B(t,x)$.

Now we turn to results regarding uniqueness, where this time we will work with a general non-uniform and time-dependent magnetic field. We first mention the result by Robert \cite{Rob97} where uniqueness for Vlasov--Poisson was shown when the initial data is compactly supported.  Then, using tools from optimal transport, Loeper made a major contribution \cite{L06} by proving that the set of solutions to the Vlasov--Poisson system with bounded microscopic density was a uniqueness class. This result was also extended to \eqref{sys:VPwB} for a constant $B$ in \cite{R21} and we discuss how to prove a similar result for a general $B$ below. Loeper's result was also generalized to less singular kernels in \cite{H09}, and very recently in \cite{I22} a new class of Wasserstein distances was introduced which improved Loeper's estimates. In \cite{M16}, Miot used some specific properties of the Vlasov--Poisson system to show uniqueness under the condition that the $L^p$ norms of the charge density grow at most linearly with respect to $p$, generalizing Loeper's uniqueness condition. This allows for solutions with unbounded charge density, more precisely with logarithmic blow-up. This result was extended to functions with charge density in Orlicz spaces in \cite{HM18}. Lastly, we underline that all these results were established in the full space setting.

This paper's second main result is to prove that Miot's uniqueness condition from \cite{M16} is also valid for \eqref{sys:VPwB} with added assumptions on the velocity moments of the initial data. In \cite{M16}, a key point was exploiting the second-order structure of the characteristics of the Vlasov--Poisson system. This explains why the uniqueness condition from \cite{M16} doesn't apply to the two-dimensional Euler model for incompressible fluids, which presents many similarities with Vlasov--Poisson, because the characteristics of the Euler model only verify a first-order ODE, whereas Loeper's condition from \cite{L06} works for both models. In our case, the main difficulty is that the added $B:=B(t,x)$ breaks the second-order structure of the Cauchy problem for the characteristics. We manage to get around this by proving that the characteristics in the magnetized case can be controlled by assuming Lipschitz regularity on $B$ in position and with the additional assumptions on the moments of the initial data mentioned above. With these additional assumptions, we deduce a new uniqueness condition which is actually the same as the sufficient condition imposed on the initial data to verify the uniqueness criterion in \cite[theorem 1.2]{M16}.

Finally, to conclude this introduction we present some interesting open problems. Naturally, we first mention the propagation of velocity moments for the magnetized Vlasov--Poisson system \eqref{sys:VPwB} with an external magnetic field that also depends on position $B:=B(t,x)$. In the proof of our main result, we explain why the approach used in this paper fails when $B:=B(t,x)$. Lastly, following the recent progress made regarding the well-posedness for the  Vlasov--Poisson system that describes the evolution of ions instead of electrons \cite{GI21two,GI21}, we could explore if the methods developed in our paper can also be applied to this ionic Vlasov--Poisson system.

\textbf{Outline of the paper:} This paper is organized as follows. We conclude this section by giving some notations and the classical a priori estimates satisfied by \eqref{sys:VPwB}. In \cref{sec:results} the main results of the paper will be presented. Then \cref{sec:prop} will be devoted to the proof of propagation of velocity moments to solutions of \eqref{sys:VPwB} in both the full space case and the periodic case using the induction argument presented above. We finish with \cref{sec:uni} where we detail our proof of uniqueness for solutions to \eqref{sys:VPwB}, highlighting how additional assumptions on the moments of the initial data allow us to control the added terms due to the external magnetic field.
\subsection{Preliminaries}
First we present the standard notation for velocity moments, for any $k\geq 0$ and $t \geq 0$ we define:
\begin{equation}\label{def:moment}
M_k(t) = \underset{0\leq s \leq t}{\sup} \iint \abs{v}^k f(s,x,v) dvdx.
\end{equation}

Now we detail the two main a priori bounds that we can deduce from \eqref{sys:VPwB}. The first bound is a direct consequence of the Vlasov equation where the coefficients are divergence-free, we have

\begin{equation}\label{eq:conser_norme}
\norme{f(t)}_p = \norme{f^{in}}_p
\end{equation} 
for all time $t$ and exponents $p \in \intervalleff{0}{+\infty}$. 

The second bound is the conservation of the energy $\mathcal{E}(t)$ of the system, with
\begin{equation}\label{def:energie}
\mathcal{E}(t):=\frac{1}{2}\iint_{\R^3\times \R^3}\abs{v}^2f(t,x,v)dxdv+\frac{1}{2}\int_{\R^3}\abs{E(t,x)}^2dx = \mathcal{E}(0) < +\infty.
\end{equation}
Furthermore, thanks to the conservation of the energy $\mathcal{E}(t)$, we have the following bounds.
\begin{lemma}
	For all $t\geq 0$, we have $M_2(t) \leq C_1$ and $\norme{\rho(t)}_\frac{5}{3} \leq C_2$ with the constants $C_1, C_2$ depending only on $\mathcal{E}(0),\norme{f^{in}}_1,\norme{f^{in}}_\infty$.
\end{lemma}

As said before, we will use the Lagrangian formulation detailed in \cite{P92}, so we define the characteristics $(X,V)$ of \eqref{sys:VPwB} which are solutions to the following Cauchy problem:

\begin{equation}\label{def:chara}
\left\{
\begin{aligned}
& \frac{d}{ds}X(s;t,x,v)=V(s;t,x,v),\\
& \frac{d}{ds}V(s;t,x,v)=E(s,X(s;t,x,v))+V(s;t,x,v) \wedge B(s),
\end{aligned}
\right.
\end{equation}
with 
\begin{equation}
\left(X(t;t,x,v),V(t;t,x,v)\right)=(x,v).
\end{equation}

Then like in \cite{P12}, we define for any $t>0$ and $\delta \in \intervalleoo{0}{t}$.

\begin{equation}\label{def:Q}
Q(t,\delta):= \sup \left\{ \int_{t-\delta}^{t} \abs{E(s,X(s;0,x,v))}ds, (x,v) \in \R^3 \times \R^3 \right\}.
\end{equation}
For the unmagnetized Vlasov--Poisson system, $Q(t,\delta)$ quantifies the evolution of the characteristics on the interval $\intervalleff{t-\delta}{t}$. However, in our context with the added magnetic field, the evolution of the velocity characteristic will be quantified both by $Q(t,\delta)$ and $\norme{B}_\infty$.

\section{Results}\label{sec:results}
\subsection{Propagation of moments}\label{subsec:Bt}
We now give the first main result of this article, which is the propagation of velocity moments of order $k>2$, extending theorem 1 in \cite{P12} to the magnetized Vlasov--Poisson system.

We specify that throughout this subsection which deals with propagation of velocity moments and in \cref{sec:prop}, to lighten the computations we set $\norme{B}_{\infty}:=\norme{B}_{L^\infty(\intervalleff{0}{T})}$.

\begin{theorem}\label{theo:main}
	Let $k_0>2, T>0, f^{in}=f^{in}(x,v)\geq 0$ a.e. with $f^{in}\in L^1\cap L^\infty (\R^3\times \R^3)$ and assume that
	\begin{equation}\label{ineq:momentini}
	\iint_{\R^3\times \R^3}\abs{v}^{k_0}f^{in}dxdv < \infty.
	\end{equation}
	Furthermore let $B:=B(t)$ verify
	 \begin{equation}\label{reguBt}
	 B \in L^\infty(\intervalleff{0}{T}).
	 \end{equation}
	Then there exists a weak solution
	\begin{equation}
	f \in C(\R_+;L^p(\R^3 \times \R^3)) \cap L^\infty(\R_+;L^p(\R^3 \times \R^3))
	\end{equation}
	$(1\leq p < +\infty)$ to the Cauchy problem for the Vlasov--Poisson system with magnetic field \eqref{sys:VPwB} in $\R^3 \times \R^3$ such that
	\begin{equation}\label{ineq:moment2}
	\underset{0\leq t \leq T}{\sup} \iint_{\R^3 \times \R^3} \abs{v}^{k_0} f(t,x,v) dvdx \leq C
	\end{equation}
	with $C$ that depends only on
	\begin{equation}\label{constants}
	T, k_0, \norme{B}_{\infty}, \mathcal{E}(0), \norme{f^{in}}_1, \norme{f^{in}}_\infty, \iint_{\R^3\times \R^3}\abs{v}^{k_0}f^{in}dxdv.
	\end{equation}
\end{theorem}

\begin{remark}
	If $f^{in}$ satisfies the assumptions of the previous theorem, then all the moments of order $k$ such that $0\leq k < k_0$ are also propagated for the solution $f$, simply because of the following H\"older inequality
	\begin{equation}\label{ineq:ptitmoments}
	\iint\abs{v}^{k}f(t,x,v)dvdx \leq \norme{f}_1^\frac{k_0-k}{k_0}\left(\iint\abs{v}^{k_0}f(t,x,v)dvdx\right)^\frac{k}{k_0}
	\end{equation}
	where we use the decomposition $\abs{v}^{k}f=f^\frac{k_0-k}{k_0} \abs{v}^{k}f^\frac{k}{k_0}$ and the exponents $p=\frac{k_0}{k_0-k}, q=\frac{k_0}{k}$.
\end{remark}

Like in \cite{LP91,R21}, we assume we have smooth solutions to conduct the proof in \cref{sec:prop}, and since the a priori estimates depend only on \eqref{constants} we can pass to the limit in the approximate Vlasov--Poisson system first introduced in \cite{AR75}. 
In fact, \cref{theo:main} will be a consequence of the main estimate in this paper which will only hold for these smooth solutions because it is an estimate on $Q$ (given in \eqref{def:Q}), which isn't necessarily well-defined for functions in Lebesgue spaces. 
We now give this estimate on $Q$.

\noindent\textbf{\large Main estimate on $Q$:}\\
\indent For all $T>0$ we have
\begin{equation}\label{ineq:mainQ}
N(T):=\underset{0\leq t \leq T}{\sup} Q(t,t) \leq C,
\end{equation}
with $C$ that depends on the constants in \eqref{constants}. In the following remark, we explain how \cref{theo:main} is a consequence of \eqref{ineq:mainQ}.

\begin{remark}\label{rem:moment}
	The estimate on propagation of velocity moments \eqref{ineq:moment2} in \cref{theo:main} follows from \eqref{ineq:mainQ} because we have:
	\begin{align*}
	\iint_{\R^3 \times \R^3} \abs{v}^k f(t,x,v) dvdx & = \iint_{\R^3 \times \R^3} \abs{V(t;0,x,v)}^k f^{in}(x,v) dvdx\\
	& \leq \iint_{\R^3 \times \R^3} (\abs{v} + N(T))^k \exp(kt\norme{B}_{\infty}) f^{in}(x,v) dv dx\\
	& \leq 2^{k-1}\exp(kt\norme{B}_{\infty})\left(\iint_{\R^3 \times \R^3} \abs{v}^k f^{in}(x,v) dvdx+N(T)^k \norme{f^{in}}_1\right)
	\end{align*}
	The first inequality above is obtained through a Gr\"onwall inequality on $\abs{V(t;0,x,v)}$, indeed thanks to \eqref{def:chara} we can write
	\begin{equation}
	V(t;0,x,v)=v + \int_{0}^{t} E(s,X(s;0,x,v))ds+\int_{0}^{t} V(s;0,x,v) \wedge B(s)ds
	\end{equation}
	which implies
	\begin{align*}
	\abs{V(t;0,x,v)} & \leq \abs{v}+Q(t,t)+\norme{B}_{\infty} \int_{0}^{t} \abs{V(s;0,x,v)}ds\\
	& \leq \abs{v}+N(T)+\norme{B}_{\infty} \int_{0}^{t} \abs{V(s;0,x,v)}ds
	\end{align*}
	This is the classical Gr\"onwall inequality which allows us to conclude that
	\begin{equation}
	\abs{V(t;0,x,v)} \leq (\abs{v}+N(T))\exp(t\norme{B}_{\infty}).
	\end{equation}
	The second inequality is just due to the fact that $2^{k-1}(1+x^k) \geq (1+x)^k$ for $x\geq 0$.
\end{remark}
We finish this subsection by discussing the periodic case. Indeed, we can generalize theorem 4 of \cite{P12}  as well as the improvement by Chen and Chen \cite{CC19} in the same way as in the full space problem. As is explained in \cite{P12}, on the torus we need a stronger assumption on the moments because of the weaker dispersion properties of the periodic magnetized Vlasov--Poisson system. In essence, this translates to the fact that charged particles can't go to infinity but have to come back and this physical property isn't offset by the presence of an external magnetic field. In practice, this means we can only show propagation of velocity moments of order $k>3$ like in \cite{CC19}.
\begin{theorem}[Propagation of moments in $\mathbb{T}^3$ with $B:=B(t)$]\label{theo:mainp}
	Let $k_0>3, T>0, f^{in}=f^{in}(x,v)\geq 0$ a.e. with $f^{in}\in L^1\cap L^\infty (\mathbb{T}^3\times \R^3)$ and assume that
	\begin{equation}\label{ineq:momentinip}
	\iint_{\mathbb{T}^3\times \R^3}\abs{v}^{k_0}f^{in}dxdv < \infty.
	\end{equation}
	Furthermore let $B:=B(t)$ verify \eqref{reguBt}. Then there exists a weak solution
	\begin{equation}
	f \in C(\R_+;L^p(\mathbb{T}^3 \times \R^3)) \cap L^\infty(\R_+;L^p(\mathbb{T}^3 \times \R^3))
	\end{equation}
	$(1\leq p < +\infty)$ to the Cauchy problem for the Vlasov--Poisson system with magnetic field \eqref{sys:VPwB} in $\mathbb{T}^3 \times \R^3$ such that
	\begin{equation}
	\underset{0\leq t \leq T}{\sup} \iint_{\R^3 \times \R^3} \abs{v}^{k_0} f(t,x,v) dvdx \leq C
	\end{equation}
	with $C$ that depends only on
	\begin{equation}
	T, k_0, \norme{B}_{\infty}, \mathcal{E}(0), \norme{f^{in}}_1, \norme{f^{in}}_\infty, \iint_{\R^3\times \R^3}\abs{v}^{k_0}f^{in}dxdv.
	\end{equation}
\end{theorem}

Just like for \cref{theo:main}, we show propagation of moments in the periodic case using a regularization of the magnetized Vlasov--Poisson system in $\mathbb{T}^3$. This is made possible by the  \cite{battrein93}, where the existence of weak solutions to the periodic Vlasov--Poisson system is proved using such a regularization of the system.

\subsection{Uniqueness}
We remind the reader that for results related to uniqueness we consider a general magnetic field that can depend on both time and position.  Furthermore, we will assume that the magnetic field has Lipschitz regularity in position. Hence, for all $T>0$ we have:
\begin{equation}\label{reguniB}
B \in L^\infty\left(\intervalleff{0}{T}, W^{1,\infty}(\R^3)\right).
\end{equation}
Throughout this subsection which deals with uniqueness and in \cref{sec:uni}, to lighten the computations we set $\norme{B}_{\infty}:=\norme{B}_{L^\infty\left(\intervalleff{0}{T}, W^{1,\infty}(\R^3)\right)}$ and $\norme{\nabla B}_\infty:=\norme{\nabla B}_{L^\infty\left(\intervalleff{0}{T},\R^3\right)}$.

As said above, Loeper's uniqueness result \cite{L06}  was extended to the Vlasov--Poisson system with constant magnetic field by the author in \cite{R21}. Now we show that Loeper's approach can also be generalized for $B:=B(t;x)$. However, with a general magnetic field we require extra regularity on the initial data $f^{in}$, as we highlight in the following theorem:
\begin{theorem}\label{theo:uniloeper}
	Let $T>0$, let $f^{in}=f^{in}(x,v)\geq 0$ a.e. with $f^{in}\in L^1\cap L^\infty (\R^3\times \R^3)$ and assume that the magnetic field $B:=B(t,x)$ verifies \eqref{reguniB}. Assume further that
	\begin{equation}\label{ineq:uniloepermomentini}
	\iint_{\R^3\times \R^3}\abs{v}^{6}f^{in}dxdv < \infty \mbox{ and } \iint_{\R^3\times \R^3}\abs{x}^{4}f^{in}dxdv < \infty.
	\end{equation}
	
	Then there exists at most one weak solution to \eqref{sys:VPwB} such that 
	\begin{equation}\label{rholinfty}
	\rho \in L^\infty(\intervalleff{0}{T} \times \R^3_x).
	\end{equation}
\end{theorem}
However, to exploit this result, it would be interesting to give a condition on the initial data that would guarantee \eqref{rholinfty}. Hence, in the next proposition, we give an explicit condition which implies the boundedness of $\rho$.
\begin{proposition} 
	Let $B$ verify \eqref{reguniB} and let $f^{in}$ satisfy the assumptions of \cref{theo:main} with $k_0>6$. We also assume that $f^{in}$ is such that for all $R>0$ and $T>0$
	\begin{equation}\label{eq:boundrho}
	g_{R}(t,x,v) \in L^\infty(\intervalleff{0}{T}\times \R^3_x,L^1(\R^3_v)),
	\end{equation}
	where
	\begin{equation}\label{eq:boundrho1}
	g_{R}(t,x,v)=\underset{(y,w) \in S_{t,x,v,R}}{\sup} f^{in}(y+vt,w) 
	\end{equation}
	with
	\begin{equation}\label{eq:boundrho2}
	S_{t,x,v,R}=\left\{(y,w) \colon \abs{y-x}\leq (R+\norme{B}_\infty \abs{v})t^2e^{\norme{B}_\infty t},\abs{w-v}\leq (R+\norme{B}_\infty \abs{v})te^{\norme{B}_\infty t}\right\}.
	\end{equation}
	Then any weak solution $f$ to \eqref{sys:VPwB} with initial data $f^{in}$ verifies
	\begin{equation*}
	\rho \in L^\infty(\intervalleff{0}{T} \times \R^3_x)
	\end{equation*}
	for all $T>0$.
\end{proposition}
This proposition was shown in \cite[proposition 2.7]{R21} in the case of a constant magnetic field and remains unchanged when we take a general $B:=B(t,x)$.

\begin{remark}
	Just like in the unmagnetized case, the above condition \eqref{eq:boundrho} is satisfied when the initial data decays sufficiently fast in the velocity variable, or more precisely for functions $f^{in}$ that satisfy
	\begin{equation}
	f^{in}(x,v) \leq \frac{C}{1+\abs{v}^p} \text{ with } p>3.
	\end{equation}
	Indeed, if the initial data verifies the above bound then we trivially have that 
	\begin{equation*}
	g_{R}(t,x,v) \leq \frac{C}{1+\left(\abs{v}-(R+\norme{B}_{\infty} \abs{v})te^{\norme{B}_{\infty} t}\right)^p}
	\end{equation*}
	and the function on the right hand side of the inequality is in $L^1(\R^3_v)$ uniformly in time and space because $p>3$.
\end{remark}

Now we present a theorem which is the second main result of this paper, where we show that the uniqueness criterion proved in \cite[theorems 1.1 and 1.2]{M16} also applies to \eqref{sys:VPwB} with $B$ verifying \eqref{reguniB}, generalizing \cref{theo:uniloeper} because it allows for solutions with unbounded charge density.
\begin{theorem}\label{theo:uni}
	Let $T>0$ and $f^{in} \geq 0$ a.e. with $f^{in} \in L^1\cap L^\infty(\R^3 \times \R^3)$, assume further that $B$ verifies \eqref{reguniB} and that $f^{in}$ satisfies
%
	\begin{equation}\label{ineq:unimoment}
	\forall k \geq 1, \quad \iint_{\R^3 \times \R^3} \abs{v}^k f^{in}(x,v)dxdv \leq (C_0 k)^\frac{k}{3},
	\end{equation}
	for some constant $C_0$ independent of $k$. 
	
	Then there exists at most one weak solution to \eqref{sys:VPwB}, and any weak solution $f$ with initial data $f^{in}$ verifies
	\begin{equation}\label{ineq:uni}
	\underset{\intervalleff{0}{T}}{\sup} \, \underset{p \geq 1}{\sup} \, \frac{\norme{\rho(t)}_p}{p} < +\infty.
	\end{equation}
\end{theorem}
\begin{remark}
	In our framework, an important difference with \cite{M16} is that the uniqueness criterion isn't given by the inequality on the charge density \eqref{ineq:uni} but rather the stronger assumption on the moments of the solution \eqref{ineq:unimoment}.
\end{remark}

As mentioned above, the assumptions of \cref{theo:uni} are less restrictive than the condition \eqref{rholinfty} and thus allow us to consider initial data with unbounded charge density. This result is illustrated by the following theorem (\cite[theorem 1.3]{M16}):
\begin{theorem}[Miot, \cite{M16}]
	There exists $f^{in} \geq 0$ a.e. such that $f^{in} \in L^1\cap L^\infty(\R^3 \times \R^3)$ satisfying the assumptions of \cref{theo:uni} and such that
	\begin{equation}
	\rho_0(x)= \frac{4\pi}{3} \ln_-(\abs{x}), \quad \forall x \in \R^3,
	\end{equation}
	where $\ln_-=\max(-\ln(x),0)$ is the negative part of the function $\ln$.
\end{theorem}
\begin{remark}
If we assume that the magnetic field is uniform $B:=B(t)$, then the proofs of theorems \ref{theo:uniloeper} and \ref{theo:uni} are greatly simplified. We can in fact show that Loeper's and Miot's uniqueness criteria, which correspond to the conditions \eqref{rholinfty} and \eqref{ineq:uni}, are also valid for \eqref{sys:VPwB}. This means that we don't require the extra regularity on the initial data \eqref{ineq:uniloepermomentini} or \eqref{ineq:unimoment} to formulate the uniqueness criteria.
\end{remark}
In \cref{sec:uni}, we will detail the proof of \cref{theo:uni} first because it is the main result of this section. Then we will present the proof of \cref{theo:uniloeper} where ingredients from the proof of \cref{theo:uni} are used, notably the estimate on the velocity characteristic. However in \cref{theo:uniloeper} we require a condition on the space moment of the initial data \eqref{ineq:uniloepermomentini} which isn't the case in \cref{theo:uni}.

Finally, with regards to uniqueness for Vlasov--Poisson, it would be very interesting to see if the conditions found in \cite{L06,M16} could be adapted to the periodic case.

\section{Propagation of velocity moments}\label{sec:prop}
In this section, we shall denote by $C$ a constant that can change from one line to another but that only depends on
\begin{equation}\label{conserved}
\mathcal{E}(0), \norme{f^{in}}_1, \norme{f^{in}}_\infty.
\end{equation}
As mentioned above, the whole proof is conducted using smooth functions.

We consider $k>2$ and $\varepsilon >0$ small enough, say $\varepsilon \in \intervalleoo{0}{\varepsilon_0}$ with $\varepsilon_0 \leq \frac{(k-2)}{2k}$.
As said in the introduction, the main difference with the analysis in \cite{P12} is that we're going to show propagation of moments for all time by using an induction argument using the cyclotron period $T_c=\norme{B}_{\infty}^{-1}$. We begin with the initialization, so we're first going consider to $T>0$ with $T \leq T_B$, where $T_B$ is the unique real number such that
\begin{equation}\label{def:TB}
T_B \in \R_+^* \text{ and } \norme{B}_{\infty}  T_B \exp(T_B \norme{B}_{\infty})=a,
\end{equation}
with $a >0$. In our method, since we can only obtain estimates on $Q$ for $T_B \ll T_c$, we just need $a$ small enough so we set $a=2^{-10}$.

Thus, we show propagation of velocity moments on $\intervalleff{0}{T_B}$ using the following result.
\begin{proposition}\label{prop:TB}
	For all $T>0$ such that $T \leq T_B$, \eqref{ineq:mainQ} is verified. More precisely we have the following estimate on $Q(t,t)$ for all $0\leq t \leq T$
	\begin{equation}
	Q(t,t) \leq C \exp(T\norme{B}_{\infty})^{\frac{2}{5}}(T^\frac{1}{2}+T^\frac{7}{5})
	\end{equation}
	with $C$ that only depends on 
	\begin{equation*}
	k, \mathcal{E}(0), \norme{f^{in}}_1, \norme{f^{in}}_\infty, M_k(0).
	\end{equation*}
\end{proposition}
\begin{remark}
	This estimate is the analogous of the estimate (13) in \cite{P12}. In our magnetized framework, we only manage to generalize this result up to the time $T_B$.
\end{remark}

The following section will be devoted to the proof of this proposition, and just like in \cite{P12} the proof is done in three steps which correspond to \cref{prop:Q1}, \cref{prop:Q2} and \cref{prop:Q3}.
\subsection{The case $T \leq T_B$}
\begin{proposition}\label{prop:Q1}
	For any $0 \leq \delta \leq t \leq T \leq T_B$ we have:
	\begin{equation}\label{ineq:Q}
	Q(t,\delta) \leq C\left( \delta 
	Q(t,\delta)^\frac{4}{3} +\delta^\frac{1}{2}(1+M_{2+\varepsilon}(T))^\frac{1}{2}\right)
	\end{equation}
\end{proposition}
\begin{proof} Let $(t,x_*,v_*) \in \intervalleff{0}{T}\times \R^3 \times \R^3$ and set $(X_*,V_*)(s)=(X,V)(s;t,x_*,v_*)$. For any $\delta \in \intervalleff{0}{t}$ we have by definition of $E$
\begin{equation*}
\int_{t-\delta}^{t} \abs{E(s,X(s;t,x_*,v_*))}ds \leq \int_{t-\delta}^{t} \int \frac{\rho(s,x)dx}{4 \pi \abs{x-X_*(s)}^2}ds
\end{equation*}
Our objective in the rest of this section will be to estimate the integral:
\begin{equation}\label{def:I*}
I_*(t,\delta):=\int_{t-\delta}^{t} \int\frac{\rho(s,x)dx}{ \abs{x-X_*(s)}^2}ds=\int_{t-\delta}^{t} \iint\frac{f(s,x,v)dvdx}{ \abs{x-X_*(s)}^2}ds
\end{equation}
Now we will use a procedure that is inspired from \cite{S91} which consists in splitting $\intervalleff{t-\delta}{t}\times \R^3 \times \R^3$ into three parts. Here the partition is slightly different because, following \cite{P12}, we introduce $\varepsilon >0$.
\begin{align*}
&G=\left\{ (s,x,v): \min(\abs{v},\abs{v-V_*(s)}) < P\right\}, \\
&B=\left\{ (s,x,v): \abs{x-X_*(s)} \leq \Lambda_\varepsilon(s,v)\right\}\backslash G, \\
&U=\intervalleff{t-\delta}{t}\times \R^3 \times \R^3 \backslash (G \cup B),
\end{align*}
with 
\begin{equation}\label{def:P}
P=2^{10}Q(t,\delta)\exp(\delta \norme{B}_{\infty}) \text{ and } \Lambda_\varepsilon(s,v)=L(1+\abs{v}^{2+\varepsilon})^{-1}\abs{v-V_*(s)}^{-1}
\end{equation}
and $L>0$ to be fixed later. The main difference here with \cite{P12} is the definition of $P$, because the added magnetic field modifies the evolution of the characteristic in velocity $V(s)$. Furthermore, we take the same numerical constant $2^{10}$ in the definition of $P$ as in \cite{P12} is (in truth this constant just needs to be large enough). Using obvious notations, we write $I_*=I_*^G+I_*^B+I_*^U$. The first two integrals will be more straightforward to estimate than $I_*^U$, which involves the set $U$, which we will call ugly set following \cite{S91}.

The first two contributions $I_*^G,I_*^B$ are treated the same in both magnetized and unmagnetized cases, simply because the modifications made to the sets $G,B$ to take into account the added magnetic field don't change the computations required to estimate $I_*^G$ and $I_*^B$. We succinctly present how to control both integrals following the calculations from \cite{P12}. The first bound is obtained by using a standard functional inequality.

For $\kappa \in L^\infty(\R^3) \cap L^\frac{5}{3}(\R^3)$ we have
\begin{equation}\label{ineqfunctio}
\norme{\kappa \ast \abs{\cdot}^{-2}}_\infty \leq c \norme{\kappa}^\frac{5}{9}_\frac{5}{3} \norme{\kappa}^\frac{4}{9}_\infty,
\end{equation}
with $c$ a numerical constant.

We apply \eqref{ineqfunctio} to the quantity:
\begin{equation}
\rho_G(s,x)=\int_{B(0,P) \cup B(V_*(s),P)} f(s,x,v)dv \leq \rho(s,x),
\end{equation}
which implies the following control on $I_*^G(t,\delta)$:
\begin{equation}\label{ineq:G}
I_*^G(t,\delta)\leq C( \delta P^\frac{4}{3}). 
\end{equation}
To estimate the contribution on $B$, we first integrate in the space variable using a spherical change of variable.
\begin{align*}
I_*^B(t,\delta) & \leq \int_{t-\delta}^{t} \int_v \int_{\abs{x-X_*(s)} \leq \Lambda_\varepsilon(s,v)} \frac{f(s,x,v)dx}{ \abs{x-X_*(s)}^2} dvds,\\
& \leq \int_{t-\delta}^{t} \int_v \left(4\pi \int_{0}^{\Lambda_\varepsilon(s,v)} \frac{\norme{f^{in}}_\infty}{r^2} r^2 dr\right)dvds,\\
& \leq C \int_{t-\delta}^{t} \int_v \frac{L}{(1+\abs{v}^{2+\varepsilon})\abs{v-V_*(s)}}dvds,\\
& \leq C \int_{t-\delta}^{t} \left(\int_{\abs{v} \leq \abs{v-V_*(s)}} \frac{L}{(1+\abs{v}^{2+\varepsilon})\abs{v}}dv +\int_{\abs{v} > \abs{v-V_*(s)}} \frac{L}{(1+\abs{v-V_*(s)}^{2+\varepsilon})\abs{v-V_*(s)}}dv\right) ds,\\
& \leq C \int_{t-\delta}^{t} \int_v \frac{L}{(1+\abs{v}^{2+\varepsilon})\abs{v}}dvds,\\
& \leq C \int_{t-\delta}^{t} \int_{\R} \frac{L r}{(1+r^{2+\varepsilon})}drds,\\
& \leq C\delta L.
\end{align*}
The last contribution $I_*^U(t,\delta)$ can be written
\begin{equation}
I_*^U(t,\delta)=\int_{t-\delta}^{t} \iint \frac{f(s,x,v)\mathbf{1}_U(s,x,v)}{\abs{x-X_*(s)}^2}dvdxds=\iint \int_{t-\delta}^{t} \frac{\mathbf{1}_U(s,X(s),V(s))}{\abs{X(s)-X_*(s)}^2}dsf(t,x,v)dvdx
\end{equation}
where we have the obvious notation $(X,V)(s)=(X,V)(s;t,x,v)$. Estimating this quantity is difficult and will occupy us for the rest of the proof of \cref{prop:Q1}.

The following lemma is very important in our proof because it highlights why we need to use the induction procedure mentioned above. In the unmagnetized case, we estimate $I_*^U$ by noticing that because of the definition of $U$, the characteristic $V(s)$ stays close to $v$ on $\intervalleff{t-\delta}{t}$ because $v$ is large compared to $P$ and $P$ is much larger $Q(t,\delta)$ which quantifies the total variation of $V(s)$ on $\intervalleff{t-\delta}{t}$. However in the magnetized case, this stays true only under the condition \eqref{def:TB} because if the magnetic field is large than the variations of $V(s)$ on $\intervalleff{t-\delta}{t}$ can also be very large compared to $P$. 
\begin{lemma}\label{lem:controlV}
Let $s_1 \in \intervalleff{t-\delta}{t}$ such that $(s_1,X(s_1),V(s_1)) \in U$, then for all $s \in \intervalleff{t-\delta}{t}$ we have
\begin{equation}\label{ineq:V}
2^{-1} \abs{v} \leq \abs{V(s)} \leq 2 \abs{v},
\end{equation}
and 
\begin{equation}\label{ineq:V-V*}
2^{-1} \abs{v-v_*} \leq \abs{V(s)-V_*(s)} \leq 2 \abs{v-v_*}.
\end{equation}
\end{lemma}
\begin{proof}
First, because of the definition of $U$, we can write
\begin{equation}\label{ineq:V1P}
\min(\abs{V(s_1)},\abs{V(s_1)-V_*(s_1)})\geq P.
\end{equation} 
Let's start by proving the first bound \eqref{ineq:V}, thanks to \eqref{def:chara} we have for all $s \in \intervalleff{t-\delta}{t}$
\begin{equation*}
V(s)-V(s_1)=\int_{s_1}^{s} E(\tau,X(\tau))d\tau+\int_{s_1}^{s}V(\tau) \wedge B(\tau,X(\tau))d\tau
\end{equation*}
Furthermore, one of the properties of the characteristics is that we have for all $(x,v) \in \R^3 \times \R^3$ and $\tau, t \in \R_+$
\begin{equation*}
X(\tau;t,x,v)=X(\tau;0,X(0;t,x,v),V(0;t,x,v))
\end{equation*}
and also that the function $(x,v) \mapsto \left(X(\tau;t,x,v),V(\tau;t,x,v)\right)$ is a $C^1$-diffeomorphism which means that 
\begin{align*}
\sup \left\{ \int_{t-\delta}^{t} \abs{E(s,X(s;t,x,v))}ds, (x,v) \in \R^3 \times \R^3 \right\}& =\sup \left\{ \int_{t-\delta}^{t} \abs{E(s,X(s;0,x,v))}ds, (x,v) \in \R^3 \times \R^3 \right\}\\
&=Q(t,\delta).
\end{align*}
Thus we can write
\begin{align*}
\abs{V(s)-V(s_1)} \leq Q(t,\delta)+\norme{B}_{\infty}\left(\delta\norme{V(s_1)}+\abs{\int_{s_1}^{s}\abs{V(\tau)-V(s_1)}d\tau}\right)
\end{align*}
which is a Gr\"onwall inequality, so we finally have for all $s \in \intervalleff{t-\delta}{t}$
\begin{equation}\label{ineq:V-V1}
\abs{\abs{V(s)}-\abs{V(s_1)}} \leq \abs{V(s)-V(s_1)} \leq \left(Q(t,\delta)+\norme{B}_{\infty}\delta\abs{V(s_1)}\right)\exp(\delta \norme{B}_{\infty}).
\end{equation}
This last inequality highlights the main difference with the unmagnetized case, indeed when we use $V(s_1)$ as a reference point to quantify the variation of $V(s)$, we see that the added term $\norme{B}_{\infty}\delta\abs{V(s_1)}\exp(\delta \norme{B}_{\infty})$, which is just the added variation of the velocity characteristic resulting from the magnetic field, is potentially unbounded. This is due to the fact that even if $0 \leq \delta \leq T$, $\norme{B}_{\infty}$ is potentially large. This is the reason we introduce the time $T_B$ which depends on the cyclotron frequency $\norme{B}_{\infty}$.

Now using this last inequality, thanks to the relation between $P$ and $Q(t,\delta)$ given in \eqref{def:P}, to \eqref{ineq:V1P}, and to \eqref{def:TB} (because $t\leq T_B$) we have
\begin{align*}
\abs{V(s)} & \leq \abs{V(s_1)}(1+\norme{B}_{\infty}\delta\exp(\delta \norme{B}_{\infty}))+2^{-10}P\\
& \leq \abs{V(s_1)}\left(1+2^{-10}+2^{-10}\right)=\abs{V(s_1)}\left(1+2^{-9}\right)
\end{align*}
and using the same relations but this time for $-\abs{V(s_1)}$ we can write
\begin{align*}
\abs{V(s_1)}\left(1-2^{-10}-2^{-10}\right) & \leq \abs{V(s_1)}\left(1-\norme{B}_{\infty}\delta\exp(\delta \norme{B}_{\infty})\right)-2^{-10}P\\
& \leq \abs{V(s_1)}\left(1-\norme{B}_{\infty}\delta\exp(\delta \norme{B}_{\infty})\right)-Q(t,\delta)\exp(\delta \norme{B}_{\infty})\\
& \leq \abs{V(s)}
\end{align*}
These inequalities are valid for all $s \in \intervalleff{t-\delta}{t}$ and so in particular for $s=t$. And so we can write
\begin{equation}
2^{-1} \abs{V(s)} \leq \abs{V(s)}\frac{1-2^{-9}}{1+2^{-9}} \leq \abs{v} \leq \abs{V(s)}\frac{1+2^{-9}}{1-2^{-9}} \leq 2\abs{V(s)}
\end{equation}
which is equivalent to \eqref{ineq:V}. The proof for \eqref{ineq:V-V*} is very similar so we won't detail it.
\end{proof}
\begin{remark}[The velocity characteristic inequality \eqref{ineq:V-V*} is false when $B:=B(t,x)$]\label{rem:Btxwrong}
	We try to apply the same analysis as above to see if \eqref{ineq:V-V*} is true when $B:=B(t,x)$ is a bounded external magnetic field that also depends on position. Like for \eqref{ineq:V}, we try to write a Gr\"onwall inequality but this time on $Z(s)=\abs{(V(s)-V_*(s))-(V(s_1)-V_*(s_1))}$.
	\begin{align*}
	(V(s)-V_*(s))-(V(s_1)-V_*(s_1))&=\int_{s_1}^{s} E(\tau,X(\tau))d\tau+\int_{s_1}^{s}V(\tau) \wedge B(\tau,X(\tau))d\tau\\
	&-\int_{s_1}^{s} E(\tau,X_*(\tau))d\tau-\int_{s_1}^{s}V_*(\tau) \wedge B(\tau,X_*(\tau))d\tau.
	\end{align*}
	This allows us to write
	\begin{align*}	
	Z(s) &\leq 2 Q(t,\delta) + \abs{\int_{s_1}^{s}\left(V(\tau) \wedge B(\tau,X(\tau))-V_*(\tau) \wedge B(\tau,X_*(\tau))\right)d\tau}\\
	& \leq 2 Q(t,\delta) +\int_{s_1}^{s}\abs{V(\tau)\wedge (B(\tau,X(\tau))-B(\tau,X_*(\tau)))}+\abs{(V_*(\tau)-V(\tau))\wedge B(\tau,X_*(\tau))}d\tau\\
	& \leq 2 Q(t,\delta) + 2\delta \norme{B}_{\infty} 2\abs{V(s_1)}\\
	& + \norme{B}_{\infty}\left(\delta\abs{(V_*(s_1)-V(s_1))}+\int_{s_1}^{s}\abs{(V(\tau)-V_*(\tau))-(V(s_1)-V_*(s_1))}d\tau\right)
	\end{align*}
	where in the second term in the last inequality we used the bound $\abs{V(s)} \leq \abs{V(s_1)}(1+2^{-9}) \leq 2 \abs{V(s_1)}$ that we established just before.
	
	Thus we have our Gr\"onwall inequality on $Z(s)$ which gives us
	\begin{equation}\label{ineq:Z}
	Z(s) \leq (2 Q(t,\delta) + 4\delta \norme{B}_{\infty} \abs{V(s_1)}+\norme{B}_{\infty}\delta\abs{(V_*(s_1)-V(s_1))})\exp(\delta \norme{B}_{\infty}).
	\end{equation}
	We notice the term $4\delta \norme{B}_{\infty} \abs{V(s_1)}$, this means that the variation of $\abs{V-V_*}$ on the ugly set $U$ depends also on $\abs{V}$ and not only on $Q(t,\delta)$ and $\abs{V-V_*}$. In practice, this means we can only obtain the following estimate when $B:=B(t,x)$
	\begin{equation}
	2^{-1} \left(\abs{v-v_*}+\abs{v}\right) \leq \abs{V(s)-V_*(s)}+\abs{V(s)} \leq 2 (\abs{v-v_*}+\abs{v}).
	\end{equation}
	Since the analysis carried out in \cite{P12} by integrating in time to estimate $I_*^U$ is done in a very fine way, the inequalities above are not enough to obtain an analogous result when $B:=B(t,x)$.
\end{remark}
\begin{lemma}
	For any $(x,v) \in \R^6$ we have
	\begin{equation}\label{ineq:1U}
	\int_{t-\delta}^{t} \frac{\mathbf{1}_U(s,X(s),V(s))}{\abs{X(s)-X_*(s)}^2}ds \leq C\left( \frac{1+\abs{v}^{2+\varepsilon}}{L}\right).
	\end{equation}
\end{lemma}
\begin{proof}
If $(s,X(s),V(s)) \notin U$ for all $s \in \intervalleff{t-\delta}{t}$ then the estimate \eqref{ineq:1U} is verified.
Now we assume that there exists $s_1 \in \intervalleff{t-\delta}{t}$ such that $(s_1,X(s_1),V(s_1)) \in U$, then thanks to \cref{lem:controlV} we can write
\begin{equation}
\Lambda_\varepsilon(s,V(s)) \geq L(1+(2\abs{v})^{2+\varepsilon})^{-1}(2\abs{v-v_*})^{-1} \geq 2^{-3-\varepsilon}\Lambda_\varepsilon(t,v)
\end{equation}
and hence
\begin{equation}\label{ineq:1Uh}
\frac{\mathbf{1}_U(s,X(s),V(s))}{\abs{X(s)-X_*(s)}^2} \leq \frac{\mathbf{1}_{\R^3 \backslash B(X_*(s),2^{-3-\varepsilon}\Lambda_\varepsilon(t,v))}(X(s))}{\abs{X(s)-X_*(s)}^2} \leq h(\abs{Y(s)}),
\end{equation}
where $Y(s)=X(s)-X_*(s)$ and $h(u)=\min(\abs{u}^{-2},4^{3+\varepsilon}\Lambda_\varepsilon(t,v)^{-2})$. Since $h$ is a non-increasing function, we look for a lower bound on $\abs{Y(s)}$.

\noindent
For any $s_0 \in \intervalleff{t-\delta}{t}$ we have, thanks to \eqref{ineq:V-V*}
\begin{align*}
\abs{Y(s)} & \geq \abs{Y(s_0)+(s-s_0)Y'(s_0)}-\abs{\int_{s_0}^{s}(s-u)Y''(u)du}\\
& \geq \abs{Y(s_0)+(s-s_0)Y'(s_0)}-2\abs{s-s_0}(Q(t,\delta)+\delta\norme{B}_{\infty}\abs{v-v_*}).
\end{align*}
Now we consider $s=s_0$ that minimizes $\abs{Y(s)}^2$ when $s \in \intervalleff{t-\delta}{t}$, then this implies $(s-s_0)Y(s_0)\cdot Y'(s_0) \geq 0$ and so
\begin{equation}
\abs{Y(s_0)+(s-s_0)Y'(s_0)}^2\geq \abs{Y'(s_0)}^2\abs{s-s_0}^2
\end{equation}
and thanks to \eqref{ineq:V-V*} we get $\abs{Y'(s_0)} \geq 2^{-1}\abs{v-v_*}$ and when we evaluate \eqref{ineq:V-V*} in $s_1$ this also yields $Q(t,\delta) \leq 2^{-9}\abs{v-v_*}\exp(-\delta \norme{B}_{\infty})$ so we have
\begin{align*}
\abs{Y'(s_0)}-2(Q(t,\delta)+\delta\norme{B}_{\infty}\abs{v-v_*}) & \geq \abs{v-v_*}(2^{-1}-2^{-8}\abs{v-v_*}\exp(-\delta \norme{B}_{\infty})\\
&-2\delta\norme{B}_{\infty})\\
& \geq \abs{v-v_*}(2^{-1}-2^{-8}-2\delta\norme{B}_{\infty}).
\end{align*}
We need the quantity $(2^{-1}-2^{-8}-2\delta\norme{B}_{\infty})$ to be strictly positive and so once again we need the condition \eqref{def:TB} for $\delta\norme{B}_{\infty}$ to be small and this inequality to be verified. Now we have
$\abs{Y(s)}\geq \alpha \abs{v-v_*} \abs{s-s_0}$ with $\alpha > 0$. Just as in \cite{P12}, we bring this inequality into \eqref{ineq:1Uh}, integrate with respect to the time variable and estimate the integral as follows to obtain \eqref{ineq:1U}:
\begin{align*}
\int_{t-\delta}^{t}h(\abs{Y(s)})ds &\leq \int_{t-\delta}^{t}h(\alpha \abs{v-v_*} \abs{s-s_0})ds,\\
& \leq \int_{0}^{+\infty} h(\alpha \abs{v-v_*} r)dr,\\
& = (\alpha \abs{v-v_*})^{-1}\int_{0}^{+\infty} h(r)dr,\\
& = (\alpha \abs{v-v_*})^{-1}\left(\int_{0}^{2^{-3-\varepsilon}\Lambda_\varepsilon(t,v)} 4^{3+\varepsilon}\Lambda_\varepsilon(t,v)^{-2} dr+ \int_{2^{-3-\varepsilon}\Lambda_\varepsilon(t,v)}^{+\infty} \frac{1}{r^2}dr\right),\\
& = (\alpha \abs{v-v_*})^{-1}(2^{3+\varepsilon}\Lambda_\varepsilon(t,v)^{-1}+2^{3+\varepsilon}\Lambda_\varepsilon(t,v)^{-1}),\\
& \leq C \left( \frac{1+\abs{v}^{2+\varepsilon}}{L}\right).
\end{align*}
\end{proof}

\vspace{2mm}
\noindent
Now integrating in $x,v$ and using the mass conservation, we finally obtain
\begin{equation}\label{ineq:ugly}
I_*^U(t,\delta) \leq C L^{-1}(1+M_{2+\varepsilon}(T)).
\end{equation}
We gather all the above estimates to conclude
\begin{align*}
I_*(t,\delta) &\leq C( \delta \left(Q(t,\delta)\exp(\delta \norme{B}_{\infty})\right)^\frac{4}{3}+\delta L+L^{-1}(1+M_{2+\varepsilon}(T)))
\\& \leq C( \delta Q(t,\delta)^\frac{4}{3}+\delta L+L^{-1}(1+M_{2+\varepsilon}(T)))
\end{align*}
where the last inequality is justified by the fact that thanks to \eqref{def:TB} we have $\exp(\delta \norme{B}_{\infty}) \leq \exp(g(2^{-10}))$ where $g$ is the inverse of the function $x \mapsto x\exp(x)$ on $\R_+$.
We conclude in the same way as in \cite{P12}, firstly by optimizing the parameter $L$ and then by noticing that the pair $(x_*,v_*)$ is arbitrary so that we have
\begin{equation}
\sup \left\{ I_*(t,\delta), (x_*,v_*) \in \R^3 \times \R^3 \right\} \geq Q(t,\delta).
\end{equation}
Finally, we obtain \eqref{ineq:Q}. \end{proof}

The next two propositions allow us to conclude. The proof of \cref{prop:Q2} is identical to the one in \cite{P12} because it doesn't rely on the characteristics of the system, but rather on real analysis arguments. In an effort of clarity, and also because some arguments of the proof are more detailed in this paper than in \cite{P12}, we place the proof of \cref{prop:Q2} in the appendix.

\begin{proposition}\label{prop:Q2}
	For any $t \in \intervalleff{0}{T}$ with $T \leq T_B$, we have
	\begin{equation}\label{ineq:Q2}
	Q(t,t) \leq C (t^{\frac{1}{2}}+t)(1+M_{2+\varepsilon}(T))^\frac{4}{7}.
	\end{equation}
\end{proposition}
Now we state the last result necessary in the proof of \cref{prop:TB}.
\begin{proposition}\label{prop:Q3}
	There exists $\tau(\varepsilon,k)>0$ such that for any $t \in \intervalleff{0}{T}$ we have
	\begin{equation}\label{ineq:Q3}
	Q(t,t) \leq  C(1+M_k(0))^{\tau(\varepsilon,k)}\exp(T\norme{B}_{\infty})^{\frac{2}{5}}(T^\frac{1}{2}+T^\frac{7}{5}).
	\end{equation}
\end{proposition}
\begin{remark}
	We notice that we obtain the same estimate \eqref{ineq:Q3} as in \cite{P12} when the magnetic field $B$ is zero.
\end{remark}
\begin{proof}
Using the same argument as in \eqref{rem:moment}, we can write
\begin{equation}\label{ineq:Mk}
M_k(t) \leq 2^{k-1}\exp(kt\norme{B}_{\infty})\left(M_k(0)+N(T)^k \norme{f^{in}}_1\right)
\end{equation}
with $N(T)$ defined in \cref{theo:main}.

Now to obtain the desired estimate we have to bound $M_{2+\varepsilon}(t)$, which we manage with the H\"older inequality. Thus for any $t\in \intervalleff{0}{T}$ we have
\begin{equation}
\iint \abs{v}^{2+\varepsilon} f(t,x,v) dx dv \leq \left(\iint \abs{v}^{2} f(t,x,v) dx dv\right)^{\frac{2+\varepsilon-k}{2-k}}\left(\iint \abs{v}^{k} f(t,x,v) dx dv\right)^{\frac{\varepsilon}{k-2}}.
\end{equation}
With the conservation of the energy $E(t)$, this H\"older inequality implies that
\begin{equation}
M_{2+\varepsilon}(T) \leq C M_k(T)^{\frac{\varepsilon}{k-2}}
\end{equation}
and bringing this inequality into \eqref{ineq:Mk} it yields
\begin{equation}
M_{2+\varepsilon}(T) \leq C \exp({\frac{k\varepsilon}{k-2}}T\norme{B}_{\infty})\left(M_k(0)+N(T)^k \norme{f^{in}}_1\right)^{\frac{\varepsilon}{k-2}}.
\end{equation}
Now thanks to \eqref{ineq:Q2} we can deduce
\begin{align*}
M_{2+\varepsilon}(T) & \leq C \exp({\frac{k\varepsilon}{k-2}}T\norme{B}_{\infty})\left(M_k(0)+N(T)^k \norme{f^{in}}_1\right)^{\frac{\varepsilon}{k-2}}\\
& \leq C \exp({\frac{k\varepsilon}{k-2}}T\norme{B}_{\infty})\left(M_k(0)+(T^{\frac{1}{2}}+T)^k(1+M_{2+\varepsilon}(T))^\frac{4k}{7} \norme{f^{in}}_1\right)^{\frac{\varepsilon}{k-2}}\\
& \leq C \exp({\frac{k\varepsilon}{k-2}}T\norme{B}_{\infty})(1+M_k(0))^{\frac{\varepsilon}{k-2}}(T^{\frac{1}{2}}+T)^{\frac{k\varepsilon}{k-2}}(1+M_{2+\varepsilon}(T))^\frac{4k\varepsilon}{7(k-2)}
\end{align*}
Like in \cite{P12}, we write $\sigma(\varepsilon,k)=\frac{4k\varepsilon}{7(k-2)}$ and notice that if we take $\varepsilon$ small enough we have $\sigma(\varepsilon,k) \in \intervalleoo{0}{1}$. More precisely, if $\varepsilon < \varepsilon_0$ then $\sigma(\varepsilon,k) \leq \frac{2}{7}, \frac{k\varepsilon}{k-2} \leq \frac{1}{2}$ and $\frac{\varepsilon}{k-2} \leq \frac{1}{2k}$ and we find
\begin{equation}
M_{2+\varepsilon}(T) \leq C(1+M_k(0))^{\frac{1}{2k}}\exp(T\norme{B}_{\infty})^{\frac{1}{2}}(1+T)^{\frac{1}{2}}(1+M_{2+\varepsilon}(T))^\frac{2}{7},
\end{equation}
where we used $T^\frac{1}{2}+T \leq 2(1+T)$. Ignoring the constant C in the above inequality, the right-hand side term is larger than 1 which implies
\begin{equation}
(1+M_{2+\varepsilon}(T))^\frac{5}{7}\leq C(1+M_k(0))^{\frac{1}{2k}}\exp(T\norme{B}_{\infty})^{\frac{1}{2}}(1+T)^{\frac{1}{2}}
\end{equation}
which finally yields
\begin{equation}
1+M_{2+\varepsilon}(T) \leq C (1+M_k(0))^{\frac{7}{10k}}\exp(T\norme{B}_{\infty})^{\frac{7}{10}}(1+T)^{\frac{7}{10}}.
\end{equation}

Then, using \eqref{ineq:Q2} again we deduce
\begin{align*}
Q(t,t) &\leq C (T^\frac{1}{2}+T)(1+M_k(0))^{\frac{2}{5k}}\exp(T\norme{B}_{\infty})^{\frac{2}{5}}(1+T)^{\frac{2}{5}}\\
& \leq C (1+M_k(0))^{\frac{2}{5k}}\exp(T\norme{B}_{\infty})^{\frac{2}{5}}(T^\frac{1}{2}+T^\frac{7}{5}).
\end{align*}
This concludes the proof of \cref{prop:Q3} and \cref{prop:TB}. \end{proof}
\subsection{The case $T \geq T_B$}
We conclude the proof of \eqref{ineq:mainQ} by showing that $Q(t,t)$ is bounded for all time.
\begin{proposition}\label{prop:allT}
	The inequality \eqref{ineq:mainQ} is valid for all $T\geq T_B$.
\end{proposition}
\begin{proof} For all $t\in \intervalleff{0}{T}$, we write $t=nT_B+t_r$ with $n \in \NN$ and $t_r \in \intervallefo{0}{T_B}$. Since the constant $C$ in \cref{prop:TB} depends only on $T,k,\norme{B}_{\infty},\norme{f^{in}}_1,\norme{f^{in}}_\infty,\mathcal{E}(0)$ and $M_k(0)$, we can reiterate the procedure on any time interval $I_p=\intervalleff{pT_B}{(p+1)T_B}$. Indeed, $T,k$ and $\norme{B}_{\infty}$ are constants $\norme{f(t)}_1$ and $\norme{f(t)}_\infty$ are conserved in time and the energy $E(t)$ is bounded. This means we can write
\begin{align*}
Q(t,t) &\leq \sum_{p=0}^{n-1}Q((p+1)T_B,T_B)+Q(t,t_r)\\
& \leq C \sum_{p=0}^{n} (1+M_k(pT_B))^{\frac{2}{5k}}\exp(T_B\norme{B}_{\infty})^{\frac{2}{5}}(T_B^\frac{1}{2}+T_B^\frac{7}{5}).
\end{align*}
Furthermore, we can show by an immediate induction that for all $p\in \NN$ with $p\leq n$, $M_k(pT_B)$ is bounded such that
\begin{equation}
M_k(pT_B) \leq C_p(k, \norme{B}_{\infty}, \mathcal{E}(0), \norme{f^{in}}_1, \norme{f^{in}}_\infty,M_k(0)).
\end{equation}
This is just because $M_k(pT_B) \leq C_1 \Rightarrow Q((p+1)T_B,T_B) \leq C_2 \Rightarrow M_k((p+1)T_B) \leq C_3$ with $C_1,C_2,C_3$ depending on \eqref{constants}.
This concludes the proof of \cref{prop:allT} and \cref{theo:main}. 

\end{proof}
\subsection{Propagation of velocity moments in the periodic case}
We choose $f^{in}$ as in \cref{theo:mainp} with $k>3$ and  consider $T>0$ which satisfies $T\leq T_B$ with $T_B$ given by \eqref{def:TB} like in the previous section. The case $T \geq T_B$ is obtained exactly like in \cref{prop:allT}.

Following Chen and Chen's improvement \cite{CC19} of Pallard's proof \cite{P12}, where they manage to improve the minimal order of the velocity moments that propagate from $\frac{14}{3}$ to $3$, we prove the following estimate on $Q(t,t)$.

\begin{proposition}\label{prop:Q1p}
	For any $0\leq \delta \leq t \leq T \leq T_B$ and $\varepsilon>0$ we have
	\begin{equation}
	Q(t,\delta) \leq C \left(\delta Q(t,\delta)^\frac{11}{6}+\delta Q(t,\delta)^\frac{1}{2}(1+M_{3+\varepsilon}(T))^\frac{1}{2}+\delta^\frac{1}{2}(1+M_{3+\varepsilon}(T))^\frac{1}{2}\right).
	\end{equation}
\end{proposition}
\begin{proof}
To prove this estimate, following \cite{CC19,P12} and like in the previous section we decompose the phase space in three parts. In our framework with an added magnetic field, we modify the partition from \cite{CC19} like in the previous section. The sets $G,B,U$ are given by
\begin{align*}
&G:=\left\{ (s,x,v): \min(\abs{v},\abs{v-V_*(s)}) < P\right\}, \\
&B:=\left\{ (s,x,v): \abs{x-X_*(s)} \leq \Lambda_\varepsilon(s,v)\right\}\backslash G, \\
&U:=\Omega \backslash (G \cup B),
\end{align*}
where $P=12 Q(t,\delta)\exp(\delta \norme{B}_{\infty})$ is modified compared to \cite{CC19}, $\Lambda_\varepsilon(s,v)=L(1+\abs{v}^{3+\varepsilon})^{-1}$ and $L$ to be fixed like in the previous section after optimizing the final estimate. We recall that the main difference with the full space framework is with the estimate of $I_*^U$, because from the periodicity in $x$ we get
\begin{equation} 
I_*^U(t,\delta)=\iint \sum_{\alpha \in A_{x,v}}\int_{t-\delta}^{t} \frac{\mathbf{1}_U(s,X(s;t,x+\alpha,v),V(s;t,x+\alpha,v))}{\abs{X(s;t,x+\alpha,v)-X_*(s)}^2}ds f(t,x,v)dvdx,
\end{equation}
where 
\begin{equation}
A_{x,v}=\{\alpha \in \Z^3 \colon \exists s \in \intervalleff{t-\delta}{t} \, (s,X(s;t,x+\alpha,v),V(s;t,x+\alpha,v)) \in U\}.
\end{equation}
Following \cite{BR91}, we can show that $A_{x,v}$ is a finite set and more precisely that
\begin{equation}
\text{\rm card}\, A_{x,v} \lesssim \abs{v-v_*}.
\end{equation}
Furthermore, by using the same analysis as in the previous section, we can show the same lemma on the velocity characteristic.
\begin{lemma}\label{lem:controlVp}
	Let $s_1 \in \intervalleff{t-\delta}{t}$ such that $(s_1,X(s_1),V(s_1)) \in U$, then for all $s \in \intervalleff{t-\delta}{t}$ we have
	\begin{equation}\label{ineq:V-V*p}
	2^{-1} \abs{v-v_*} \leq \abs{V(s)-V_*(s)} \leq 2 \abs{v-v_*}.
	\end{equation}
\end{lemma}
This estimate allows us to proceed like in \cite{CC19} to obtain the following contribution of the ugly set $I_*^U$
\begin{equation}
\sum_{\alpha \in A_{x,v}}\int_{t-\delta}^{t} \frac{\mathbf{1}_U(s,X(s;t,x+\alpha,v),V(s;t,x+\alpha,v))}{\abs{X(s;t,x+\alpha,v)-X_*(s)}^2}ds \lesssim L^{-1}(1+\abs{v}^{3+\varepsilon})\left(\delta+\frac{1}{Q(t,\delta)}\right)
\end{equation}
which explains why we have an order $3+\varepsilon$ in the estimate of \cref{prop:Q1p}. This last estimate concludes the proof of \cref{prop:Q1p} because we can control the contributions of the two other sets exactly like in the full space case.
\end{proof}
From \cref{prop:Q1p}, we prove \cref{theo:mainp} by adapting the rest of the analysis from \cite{CC19} in the same way as in the previous section.

\section{Proofs regarding uniqueness}\label{sec:uni}
The subsections \ref{sec:uniMiot1}, \ref{sec:uniMiot2} and \ref{sec:uniMiot3} will be devoted to the proof of \cref{theo:uni} and subsection \ref{sec:uniloeper} will be devoted to the proof of \cref{theo:uniloeper}.
In this section, we shall denote by $C$ a constant that can change from one line to another but that only depends on
\begin{equation}\label{conserveduni}
\mathcal{E}(0), \norme{f^{in}}_1, \norme{f^{in}}_\infty, T, \iint \abs{v}^m f^{in}.
\end{equation}
\subsection{Proof of the estimate on the $L^p$ norms of $\rho$ \eqref{ineq:uni}}\label{sec:uniMiot1}
We consider $f^{in}$ that satisfies the assumptions of \cref{theo:uni} and let $f$ be the solution given by \cref{theo:main} with initial data $f^{in}$. By construction, we have
propagation of moments:
\begin{equation}\label{ineq:propuni}
\underset{t \in \intervalleff{0}{T}}{\sup}\iint_{\R^3 \times \R^3} \abs{v}^m f(t,x,v)dxdv < +\infty.
\end{equation}
Now thanks to a classical velocity moment inequality, we show how to control the $L^p$ norms of the charge density with velocity moments, this inequality is given by
\begin{equation}\label{ineq:rhoMk}
\norme{\rho(t)}_\frac{k+3}{3} \leq C \norme{f(t)}_\infty^\frac{k}{k+3}M_k(t)^\frac{3}{k+3}.
\end{equation}
with $C$ independent of $k$. Since we want $\rho$ to verify \eqref{ineq:uni}, this means that we need to prove
\begin{equation}
\forall k \geq 1, \underset{t \in \intervalleff{0}{T}}{\sup}(\norme{f(t)}_\infty^\frac{k}{k+3}M_k(t)^\frac{3}{k+3}) \leq Ck.
\end{equation}
Since the solution $f \in L^\infty(\intervalleff{0}{T},L^\infty(\R^3 \times \R^3))$, we finally need to show 
\begin{equation}
\forall k \geq 1, \underset{t \in \intervalleff{0}{T}}{\sup}M_k(t)^\frac{3}{k+3} \leq Ck.
\end{equation}
First, we recall that thanks to \eqref{ineq:propuni} where $m>6$ we can infer that $\rho \in L^\infty(\intervalleff{0}{T},L^p(\R^3))$ with $p=\frac{m+3}{3}>3$ and following \eqref{ineq:E2} we have $E \in L^\infty(\intervalleff{0}{T},L^\infty(\R^3))$.
Then we write
\begin{equation*}
\frac{d}{dt} \abs{V(t,x,v)}^k = k \abs{V(t,x,v)}^{k-1} \frac{\dot{V}(t,x,v) \cdot V(t,x,v)}{\abs{V(t,x,v)}},
\end{equation*}
and thanks to the bound on $E$ and the definition of the characteristics \eqref{def:chara} we can infer that for all $k>m$
\begin{align*}
\abs{V(t,x,v)}^k & \leq \abs{v}^k+k\int_{0}^{t} \abs{V(s,x,v)}^{k-1} \frac{\left(E(s,X(s,x,v))+V(s,x,v) \wedge B(s,X(s,x,v))\right) \cdot V(s,x,v)}{\abs{V(s,x,v)}} ds\\
& \leq \abs{v}^k+k\norme{E}_\infty \int_{0}^{t} \abs{V(s,x,v)}^{k-1}ds.
\end{align*}
Since the contribution of magnetic field $B$ vanishes, the following computations are the same as in the unmagnetized case \cite{M16}. In an effort to be clear, we explicit these computations nonetheless.

Integrating this last inequality with respect to $f^{in}(x,v)dxdv$ we get
\begin{equation}
M_k(t) \leq M_k(0)+k\norme{E}_\infty \int_{0}^{t} M_{k-1}(s)ds.
\end{equation}
Thus by induction we deduce that $\underset{t \in \intervalleff{0}{T}}{\sup} M_k(t)$ is finite for all $k>m$. furthermore, by another classical velocity moment inequality we obtain that
\begin{equation}
M_{k-1}(s) \leq \norme{f(s)}_1^\frac{1}{k} M_k(s)^\frac{k-1}{k}.
\end{equation}
Since $\norme{f(t)}_1$ is conserved, we get
\begin{equation}
M_k(t) \leq M_k(0)+Ck\int_{0}^{t}  M_k(s)^\frac{k-1}{k}ds,.
\end{equation}
Differentiating this inequality allows us to write
\begin{equation*}
M_k'(t) \leq C k M_k(t)^\frac{k-1}{k} \Leftrightarrow \frac{d}{dt}(M_k(t)^\frac{1}{k}) \leq C \Rightarrow \underset{t \in \intervalleff{0}{T}}{\sup} M_k(t)^\frac{1}{k} \leq M_k(0)^\frac{1}{k}+C.
\end{equation*}
By assumption on $M_k(0)$ we find for all $t \in \intervalleff{0}{T}$
\begin{equation}
M_k(t)^\frac{1}{k} \leq (C_0 k)^\frac{1}{3}+C \leq (Ck)^\frac{1}{3} \leq (Ck)^{\frac{1}{3}+\frac{1}{k}},
\end{equation}
which finally implies that
\begin{equation}
\underset{t \in \intervalleff{0}{T}}{\sup} M_k(t)^\frac{3}{k+3} \leq Ck.
\end{equation}

\subsection{Estimate on the characteristics}\label{sec:uniMiot2}
We consider two solutions $f_1, f_2 \in L^\infty(\intervalleff{0}{T},L^1\cap L^\infty(\R^3 \times \R^3))$ such that $\rho_1, \rho_2$ verify
\begin{equation}\label{regurho}
\rho_1, \rho_2 \in L^\infty(\intervalleff{0}{T},L^p(\R^3))
\end{equation}
for some $p > 3$. This regularity on $\rho_{1,2}$ is guaranteed by the condition \eqref{ineq:unimoment} thanks to the estimate \eqref{ineq:rhoMk}. Then we write $Y_1=(X_1,V_1)$ and $Y_2=(X_2,V_2)$ for the corresponding characteristics, which are both solutions to \eqref{def:chara} with $t=0$. This means we can simplify the notation and will write $Y_{i}(t;0,x,v)=Y_{i}(t,x,v)$, $i=1,2$. Regarding the existence of such characteristics, the condition \eqref{regurho} yields sufficient regularity on the electric field $E_{i}$, $i=1,2$, so that with the added regularity assumption on the magnetic field \eqref{reguniB} we can define weak characteristics thanks to theorem III.2 (section III.2) in \cite{DL89}.

Now we introduce the distance
\begin{equation}\label{def:D}
D(t)=\iint_{\R^3 \times \R^3} \abs{X_1(t,x,v)-X_2(t,x,v)}f^{in}(x,v)dxdv.
\end{equation}
From \eqref{def:chara} we can write that
\begin{equation}
\begin{aligned}
&X_1(t,x,v)-X_2(t,x,v)=\int_{0}^{t} \int_{0}^{s}E_1(\tau,X_1(\tau,x,v))-E_2(\tau,X_2(\tau,x,v))\\
& +V_1(\tau,x,v) \wedge B(\tau,X_1(\tau,x,v))-V_2(\tau,x,v) \wedge B(\tau,X_2(\tau,x,v))d\tau ds
\end{aligned}
\end{equation}
which yields that 
\begin{align*}
D(t)  &\leq \int_{0}^{t} \int_{0}^{s} \int_{\R^6}\abs{E_1(\tau,X_1(\tau,x,v))-E_2(\tau,X_2(\tau,x,v))}\\
& +\abs{V_1(\tau,x,v) \wedge B(\tau,X_1(\tau,x,v))-V_2(\tau,x,v) \wedge B(\tau,X_2(\tau,x,v))}f^{in}(x,v)dxdvd\tau ds\\
&\leq \int_{0}^{t}\int_{0}^{s}\int_{\R^6}\abs{E_1(\tau,X_1(\tau,x,v))-E_2(\tau,X_2(\tau,x,v))}f^{in}(x,v) dxdvd\tau ds\\ &+\norme{B}_\infty \int_{0}^{t}\int_{0}^{s} \int_{\R^6} \abs{V_1(\tau,x,v)-V_2(\tau,x,v)} f^{in}(x,v)dxdvd\tau ds\\
&+\int_{0}^{t}\int_{0}^{s} \int_{\R^6} \abs{V_2(\tau,x,v)}\abs{B(\tau,X_1(\tau,x,v))-B(\tau,X_2(\tau,x,v))}f^{in}(x,v) dxdvd\tau ds\\
& = I(t)+J(t)+K(t)
\end{align*}
The term $I(t)$ is the quantity estimated thanks to the method in \cite{M16}. As for the other two terms $J(t)$ and $K(t)$, since we want to use the same method as in \cite{M16} which is to exploit the fact that the characteristics of the Vlasov equation verify an ODE of order 2, we need them to be controlled by $\int_{0}^{t}\int_{0}^{s} D(\tau)^{1-\frac{3}{p}}d\tau ds$. This is true and the estimates are given in the following proposition.
\begin{proposition}\label{prop:estiuni}
	For all $t \in \intervalleff{0}{T}$ and for all $p>3$, we have the following estimates:
	\begin{equation}\label{ineq:I}
	I(t) \leq CpC_{\rho_1,\rho_2}\int_{0}^{t}\int_{0}^{s} D(\tau)^{1-\frac{3}{p}}d\tau ds
	\end{equation}
	\begin{equation}\label{ineq:J}
	K(t) \leq (CpK_B+K_{B,p})\int_{0}^{t}\int_{0}^{s} D(\tau)^{1-\frac{3}{p}}d\tau ds
	\end{equation}
	\begin{equation}\label{ineq:K}
	\begin{aligned}
	&J(t) \leq \norme{B}_\infty \int_{0}^{t}\int_{0}^{s} \int_{0}^{\tau}\left(CpC_{\rho_1,\rho_2}+(CpK_B+K_{B,p})\right)D(u)^{1-\frac{3}{p}} du d\tau ds\\
	& +\norme{B}_\infty^2 \exp(T\norme{B}_\infty)\int_{0}^{t}\int_{0}^{s} \int_{0}^{\tau} \int_{0}^{u}\left(CpC_{\rho_1,\rho_2}+(CpK_B+K_{B,p})\right)D(w)^{1-\frac{3}{p}}dw du d\tau ds.
	\end{aligned}
	\end{equation}
	with
	\begin{align*}
	C_{\rho_1,\rho_2} & =\max\left(1+\norme{\rho_1}_{L^\infty(\intervalleff{0}{T},L^p)},1+\norme{\rho_2}_{L^\infty(\intervalleff{0}{T},L^p)} \right),\\
	K_{B,p} & =2 \norme{B}_{\infty}\norme{E_2}_\infty\exp(T\norme{B}_\infty),\\
	K_B & = 2 \norme{B}_{\infty}\exp(T\norme{B}_\infty),
	\end{align*}  
	and where $C$ denotes a constant that depends only on $T,\norme{f^{in}}_\infty, \norme{f^{in}}_1$.
\end{proposition}
\begin{proof}[Proof of \cref{prop:estiuni}]
	As said above, the term $I(t)$ is the quantity estimated thanks to the method in \cite{M16}, so we treat it identically to find the estimate \eqref{ineq:I}.
	
	Let's first look at the term $K(t)$. Since $B \in L^\infty\left(\intervalleff{0}{T}, W^{1,\infty}(\R^3)\right)$ then for all $t \in \intervalleff{0}{T}$ and $\alpha \in \intervalleof{0}{1}$ 
	\begin{equation}\label{BinHolder}
	B(t) \in C^{0,\alpha}(\R^3)
	\end{equation}
	with H\"older coefficient $C_{B(t)}$ verifying $C_{B(t)} \leq \max(2\norme{B}_\infty,\norme{\nabla B}_\infty) \leq 2 \norme{B}_{\infty}$.
	
	Then we simply have for all $p>3$
	\begin{equation}
	K(t) \leq 2 \norme{B}_\infty\int_{0}^{t}\int_{0}^{s} \int_{\R^6} \abs{V_2(\tau,x,v)}\abs{X_1(\tau,x,v)-X_2(\tau,x,v)}^{1-\frac{3}{p}}f^{in}(x,v) dxdvd\tau ds
	\end{equation}
	Now we need to estimate the velocity characteristic $V_2$, and using \eqref{def:chara} we can write once again
	\begin{align*}
	\abs{V(t,x,v)} & \leq \abs{v}+\int_{0}^{t}\abs{E(s,X(s,x,v))}ds+\norme{B}_\infty\int_{0}^{t}\abs{V(s,x,v)}ds\\
	& \leq \abs{v}+T\norme{E}_\infty+\norme{B}_\infty\int_{0}^{t}\abs{V(s,x,v)}ds.
	\end{align*} 
	This classical Gr\"onwall inequality yields for all $t \in \intervalleff{0}{T}$
	\begin{equation}\label{ineq:V(t)}
	\abs{V(t,x,v)} \leq (\abs{v}+T\norme{E}_\infty)\exp(t\norme{B}_\infty).
	\end{equation}
	So that we can write
	\begin{equation}
	\begin{aligned}
	K(t) & \leq 2 \norme{B}_{\infty}\exp(T\norme{B}_\infty),\\
	&\int_{0}^{t}\int_{0}^{s} \int_{\R^6} (\abs{v}+T\norme{E_2}_\infty)\abs{X_1(\tau,x,v)-X_2(\tau,x,v)}^{1-\frac{3}{p}} f^{in}(x,v) dxdvd\tau ds,\\
	& = K_1(t)+K_2(t).
	\end{aligned}
	\end{equation}
	By applying Jensen's inequality for concave functions to $x \mapsto x^{1-\frac{3}{p}}$ we obtain
	\begin{equation}\label{ineq:K2}
	K_2(t) \leq K_{B,p}\int_{0}^{t}\int_{0}^{s} D(\tau)^{1-\frac{3}{p}}d\tau ds.
	\end{equation}
	Then we estimate $K_1(t)$ by writing $f^{in}=(f^{in})^{\frac{1}{p}}(f^{in})^{\frac{1}{p'}}$ where $\frac{1}{p}+\frac{1}{p'}=1$, so that with the H\"older inequality applied to $\abs{v}(f^{in})^{\frac{1}{p}}$ and $\abs{X_1(\tau,x,v)-X_2(\tau,x,v)}^{1-\frac{3}{p}}(f^{in})^{\frac{1}{p'}}$ with the exponents $p$ and $p'$ we have
	\begin{align*}
	K_1(t) \leq K_B \left(\int_{\R^6}\abs{v}^p f^{in}(x,v) dxdv\right)^\frac{1}{p}\int_{0}^{t}\int_{0}^{s} \left(\int_{\R^6}\abs{X_1(\tau,x,v)-X_2(\tau,x,v)}^{(1-\frac{3}{p})p'}f^{in}(x,v)dxdv\right)^\frac{1}{p'}d\tau ds.
	\end{align*}
	Using \eqref{ineq:unimoment} we have $\left(\int_{\R^6}\abs{v}^p f^{in}(x,v) dxdv\right)^\frac{1}{p} \leq (C_0 p)^\frac{1}{3} \leq Cp$. Furthermore, we can once again use the Jensen inequality because $(1-\frac{3}{p})p' =\frac{p-3}{p}\frac{p}{p-1}=\frac{p-3}{p-1}<1$, which gives us
	\begin{equation}\label{ineq:K1}
	K_1(t) \leq C pK_B \int_{0}^{t}\int_{0}^{s} D(\tau)^{1-\frac{3}{p}}d\tau ds. 
	\end{equation}
	This concludes the proof of \eqref{ineq:K}.

	To estimate the last term $J(t)$, we also use \eqref{def:chara} to obtain a Gr\"onwall inequality on $\abs{V_1(t)-V_2(t)}$, and since the computations are complicated we write $V_{1,2}(s),X_{1,2}(s)$ for the characteristics. First we write
	\begin{align*}
	\abs{V_1(t)-V_2(t)} &\leq \int_{0}^{t} \abs{E_1(s,X_1(s))-E_2(s,X_2(s))}ds+\norme{B}_\infty \int_{0}^{t} \abs{V_1(s)-V_2(s)}ds\\
	&+\int_{0}^{t} \abs{V_2(s)}\abs{B(s,X_1(s))-B(s,X_2(s))}ds.
	\end{align*}
	Now using \eqref{BinHolder} and \eqref{ineq:V(t)} we deduce
	\begin{align*}
	\abs{V_1(t)-V_2(t)} &\leq \int_{0}^{t} \abs{E_1(s,X_1(s))-E_2(s,X_2(s))}ds+\norme{B}_\infty \int_{0}^{t} \abs{V_1(s)-V_2(s)}ds\\
	&+(K_B \abs{v}+K_{B,p})\int_{0}^{t}\abs{X_1(s)-X_2(s)}^{1-\frac{3}{p}}ds
	\end{align*}
	which is just the Gr\"onwall inequality on $\abs{V_1(t)-V_2(t)}$ we were looking for and which yields
	\begin{align*}
	&\abs{V_1(t)-V_2(t)}  \leq \int_{0}^{t} \abs{E_1(s,X_1(s))-E_2(s,X_2(s))}+(K_B \abs{v}+K_{B,p})\abs{X_1(s)-X_2(s)}^{1-\frac{3}{p}}ds\\
	&+\int_{0}^{t} \left(\int_{0}^{s} \abs{E_1(\tau,X_1(\tau))-E_2(\tau,X_2(\tau))}+(K_B \abs{v}+K_{B,p})\abs{X_1(\tau)-X_2(\tau)}^{1-\frac{3}{p}}
	d\tau\right)\\
	& \times \norme{B}_\infty \exp((t-s)\norme{B}_\infty)ds.	
	\end{align*}
	
	Now we insert this inequality in the definition of $J(t)$ to obtain
	\begin{equation}
	\begin{aligned}
	& J(t) \leq \norme{B}_\infty \int_{0}^{t}\int_{0}^{s} \int_{0}^{\tau} \int_{\R^6} \left(\abs{E_1(u,X_1(u))-E_2(u,X_2(u))}+(K_B \abs{v}+K_{B,p})\abs{X_1(u)-X_2(u)}^{1-\frac{3}{p}}\right)\\
	& \times f^{in}(x,v)dx dv du d\tau ds\\
	&+\norme{B}_\infty^2 \exp(T\norme{B}_\infty) \times\\
	& \int_{0}^{t}\int_{0}^{s} \int_{0}^{\tau} \int_{0}^{u} \int_{\R^6} \left(\abs{E_1(w,X_1(w))-E_2(w,X_2(w))}+(K_B \abs{v}+K_{B,p})\abs{X_1(w)-X_2(w)}^{1-\frac{3}{p}}\right)\\& \times f^{in}(x,v)dx dv dw du d\tau ds.
	\end{aligned}
	\end{equation}
	Like previously, we can use the Jensen inequality to bound the terms $K_{B,p}\abs{X_1-X_2}^{1-\frac{3}{p}}$, the relation \eqref{ineq:I} to bound the terms $\abs{E_1-E_2}$ and the H\"older inequality used to estimate $K_1(t)$ to bound $K_B \abs{v}\abs{X_1-X_2}^{1-\frac{3}{p}}$.
	
	This gives the desired estimate \eqref{ineq:J} on $J(t)$:
	\begin{equation*}
	\begin{aligned}
	&J(t) \leq \norme{B}_\infty \int_{0}^{t}\int_{0}^{s} \int_{0}^{\tau}\left(CpC_{\rho_1,\rho_2}+(C p K_B+K_{B,p})\right)D(u)^{1-\frac{3}{p}} du d\tau ds\\
	& +\norme{B}_\infty^2 \exp(T\norme{B}_\infty)\int_{0}^{t}\int_{0}^{s} \int_{0}^{\tau} \int_{0}^{u}\left(CpC_{\rho_1,\rho_2}+(C p K_B+K_{B,p})\right)D(w)^{1-\frac{3}{p}}dw du d\tau ds.
	\end{aligned}
	\end{equation*}
\end{proof}
\subsection{A second order inequality on $D(t)$}\label{sec:uniMiot3}

We begin by looking at the dependence of $K_{B,p}$ with respect to $p$. The only term in $K_{B,p}$ which depends on $p$ is $\norme{E_2}_\infty$, and since $\rho_2 \in L^\infty(\intervalleff{0}{T},L^p(\R^3))$ with $p>3$, then we can deduce the desired $L^\infty$ bound on $E_2$ because for all $t \in \intervalleff{0}{T}$
\begin{equation}\label{ineq:Erho}
\norme{E_2(t)}_\infty \leq \norme{\mathbf{1}_{\abs{x}\geq 1} \nabla \mathcal{G}_3}_\infty \norme{\rho_2(t)}_1 + \norme{\mathbf{1}_{\abs{x}< 1} \nabla \mathcal{G}_3}_q \norme{\rho_2(t)}_p
\end{equation}
with $\frac{1}{p}+\frac{1}{q}=1$.

From this last inequality we can finally deduce 
\begin{equation}\label{ineq:E2}
\norme{E_2}_\infty \leq C\left(1+\norme{\rho_2}_{L^\infty(\intervalleff{0}{T},L^p)}\right),
\end{equation}
where $C$ depends only on $\norme{f^{in}}_1$.

Now we consider that the solutions $f_1,f_2$ verify the assumptions of \cref{theo:uni}. This means that $\max(\norme{\rho_1}_{L^\infty(\intervalleff{0}{T},L^p)},\norme{\rho_2}_{L^\infty(\intervalleff{0}{T},L^p)}) \leq Cp$ for all $p \geq 1$, and so thanks to \eqref{ineq:E2} and  \cref{prop:estiuni} we have for all $p>3$
\begin{equation}
\begin{aligned}
& D(t) \leq C_1p^2\int_{0}^{t}\int_{0}^{s} D(\tau)^{1-\frac{3}{p}}d\tau ds\\
& +C_2(p^2+p)\int_{0}^{t}\int_{0}^{s} \int_{0}^{\tau} D(u)^{1-\frac{3}{p}}du d\tau ds \\
& +C_3(p^2+p)\int_{0}^{t}\int_{0}^{s} \int_{0}^{\tau} \int_{0}^{u} D(w)^{1-\frac{3}{p}}dw du d\tau ds,
\end{aligned}
\end{equation}
where $C_1,C_2,C_3$ are constants that depend on $T,\norme{f^{in}}_\infty, \norme{f^{in}}_1,\norme{B}_\infty, \norme{\nabla B}_\infty$.

Let $\mathcal{F}(t) = \int_{0}^{t}\int_{0}^{s} D(\tau)^{1-\frac{3}{p}}d\tau ds$. Since $\mathcal{F}$ is increasing by construction and since $p>3$ we can finally conclude that

\begin{equation}
D(t) \leq Cp^2\int_{0}^{t}\int_{0}^{s} D(\tau)^{1-\frac{3}{p}}d\tau ds,
\end{equation}
with $C$ that depends on $T,\norme{f^{in}}_\infty, \norme{f^{in}}_1,\norme{B}_\infty, \norme{\nabla B}_\infty$.

Finally, we obtain the same second order differential inequality as in \cite{M16}, for all $t\in \intervalleff{0}{T}$ we have:
\begin{equation}
\mathcal{F}''(t) \leq Cp^2 \mathcal{F}(t).
\end{equation}

From this inequality, we use the same method as in \cite{M16} to conclude that for all $t\in \intervalleff{0}{T}$ we have $f_1(t)=f_2(t)$ a.e. on $\R^3 \times \R^3$. This concludes the proof of \cref{theo:uni}.

\subsection{Proof of \cref{theo:uniloeper}}\label{sec:uniloeper}

We finish this section with the proof of \cref{theo:uniloeper}, which is the extension of Loeper's result \cite{L06} to the magnetized Vlasov--Poisson system. 

Like in the \cref{theo:uni}, we require additional assumptions on the moments of $f^{in}$ to obtain uniqueness (compared to the unmagnetized case). However, these assumptions on the moments aren't as strong as in \cref{theo:uni} because the boundedness of $\rho$ is already a strong assumption.

To prove our theorem, we only need to adapt subsection 3.2 from \cite{L06}. Thus we consider two solutions of \eqref{sys:VPwB} $f_1, f_2$ with initial datum $f^{in}$ that verifies the assumptions of \cref{theo:uniloeper}. Like in the previous proof, we write the corresponding densities, electric fields, and characteristics $\rho_1,\rho_2$, $E_1,E_2$, and $Y_1(t,x,v), Y_2(t,x,v)=(X_1(t,x,v),V_1(t,x,v)),(X_2(t,x,v),V_2(t,x,v))$. To simplify the presentation, we will write $Y_i(t)$ for the characteristics. We define the following quantity $Q$:
\begin{equation}
Q(t)=\frac{1}{2} \int_{\R^6} f^{in}(x,v)\abs{Y_1(t,x,v)-Y_2(t,x,v)}^2 dxdv.
\end{equation}
Now we differentiate $Q$ (which we couldn't do with the distance $D$ \eqref{def:D}) splitting the magnetic part of the Lorentz force $V \wedge B$ like in the previous section:
\begin{align*}
\dot{Q}(t) & =\int_{\R^6} f^{in}(x,v)(Y_1(t)-Y_2(t)) \cdot \partial_t(Y_1(t)-Y_2(t))dxdv,\\
& = \int_{\R^6} f^{in}(x,v)(X_1(t)-X_2(t)) \cdot  (V_1(t)-V_2(t))dxdv\\
& + \int_{\R^6} f^{in}(x,v)(V_1(t)-V_2(t)) \cdot (E_1(t,X_1(t))-E_2(t,X_2(t)) )dxdv\\
& + \int_{\R^6} f^{in}(x,v)(V_1(t)-V_2(t)) \cdot  \left[V_2(t) \wedge (B_1(t,X_1(t))-B_2(t,X_2(t))\right]dxdv\\
& + \int_{\R^6} f^{in}(x,v)(V_1(t)-V_2(t)) \cdot \left[(V_1(t)-V_2(t)) \wedge B(t,X_1(t))\right] dxdv.
\end{align*}
First, we notice that the last term is null, which means we only need to control the second to last term (due to the added magnetic field) which we denote $P(t)$. The first term is bounded by $Q(t)$ and the second term can be estimated using the analysis from \cite{L06} and is bounded by $Q(t)\ln(\frac{1}{Q(t)})$. To control $P(t)$ we first use the bound on the velocity characteristic \eqref{ineq:V(t)}. 
\begin{align*}
P(t) &\leq \norme{B}_{W^{1,\infty}}\int_{\R^6} f^{in}(x,v)\abs{V_1(t)-V_2(t)}  \abs{V_2(t)} \abs{X_1(t)-X_2(t)}dxdv,\\
&\leq \norme{B}_{W^{1,\infty}}\int_{\R^6} f^{in}(x,v)\abs{V_1(t)-V_2(t)}  (\abs{v}+T\norme{E_2}_\infty)e^{T\norme{ B}_\infty} \abs{X_1(t)-X_2(t)}dxdv,\\
&= R(t)+S(t).
\end{align*}
We recall that since $\norme{\rho_{1,2}}_\infty \leq +\infty$ we can bound $\norme{E_{1,2}}_\infty$ thanks to \eqref{ineq:E2} and the interpolation inequality:
\begin{equation}
\norme{E_{i}}_\infty \leq C(\norme{\rho}_1,\norme{\rho}_\infty):=C_{\rho},
\end{equation}
with $i=1,2$.

This means we can simply estimate $S(t)$ with the Cauchy--Schwarz inequality applied on the functions $(f^{in})^\frac{1}{2}\abs{V_1(t)-V_2(t)}$ and $(f^{in})^\frac{1}{2}\abs{X_1(t)-X_2(t)}$. 
\begin{align*}
S(t)& \leq T C_\rho C_{B,T}\left(\int_{\R^6} f^{in}(x,v)\abs{V_1(t)-V_2(t)}^2\right)^\frac{1}{2}  \left(\int_{\R^6} f^{in}(x,v)\abs{X_1(t)-X_2(t)}^2\right)^\frac{1}{2}\\
&\leq T C_\rho C_{B,T} Q(t),
\end{align*}
with $C_{B,T}=\norme{B}_{W^{1,\infty}}e^{T\norme{ B}_\infty}$.

To control $R(t)$ we first use the Cauchy--Schwarz inequality and then the bound on the velocity characteristic \eqref{ineq:V(t)}, which also gives us a bound on the position characteristic. 
\begin{align*}
R(t) &\leq C_{B,T}\int_{\R^6} f^{in}(x,v)\abs{v}\abs{Y_1(t)-Y_2(t)}^2 dxdv\\
& \leq C_{B,T}\int_{\R^6} f^{in}(x,v)\abs{v}\abs{Y_1(t)-Y_2(t)} \left(\abs{V_1}^2+\abs{V_2}^2+\abs{X_1}^2+\abs{X_2}^2\right)^\frac{1}{2}dxdv\\
& \leq C_{B,T}Q(t)\int_{\R^6} f^{in}(x,v)\abs{v}^2 \left(\abs{V_1}^2+\abs{V_2}^2+\abs{X_1}^2+\abs{X_2}^2\right)dxdv\\
& \leq C_{B,T}Q(t)\underset{=I}{\underbrace{\int_{\R^6} f^{in}(x,v)\abs{v}^2 2\left((\abs{v}+TC_\rho)^2 e^{2T\norme{ B}_\infty}+\left(\abs{x}+T(\abs{v}+TC_\rho) e^{T\norme{ B}_\infty}\right)^2\right)dxdv}}.
\end{align*}
Thanks to the assumption \eqref{ineq:uniloepermomentini} of \cref{theo:uniloeper}, $I$ is bounded because we have 
\begin{equation}
I \leq C\left(\int f^{in}\abs{v}^6,\int f^{in}\abs{x}^4\right). 
\end{equation}

From these estimates, we conclude that
\begin{equation}\label{ineq:finloeper}
\frac{d}{dt} Q(t) \leq C Q(t)\left(1+\ln \frac{1}{Q(t)}\right)
\end{equation}
with $C:=C\left(T,\norme{B}_{W^{1,\infty}},\norme{\rho}_1,\norme{\rho}_\infty,\int f^{in}\abs{v}^6,\int f^{in}\abs{x}^4\right)$.

With this inequality we can show, using a Gr\"onwall type inequality, that $Q(0)=0 \Rightarrow Q(t)=0$ for all $t \geq 0$, which concludes the proof of \cref{theo:uniloeper}.

\section*{Appendix}\label{sec:appen}
As said above, we present a slightly more detailed version of the proof of \cref{prop:Q2} compared to the one found in \cite{P12}.

\begin{proof}[Proof of \cref{prop:Q2}]
	Let $t \in \intervalleff{0}{T}$. We note here $H=1+M_{2+\varepsilon}(T)$ and for any $\delta \in \intervalleff{0}{t}$ we define $N_1(t,\delta)=\delta Q(t,\delta)^\frac{4}{3}$
	and $N_2(t,\delta)=(\delta H)^\frac{1}{2}$ as in the left hand side of inequality \eqref{ineq:Q}. We set:
	\begin{equation}
	I=\left\{\delta \in \intervalleff{0}{t} : N_1(t,\delta) \geq N_2(t,\delta)\right\}.
	\end{equation}
	First let's suppose that $I$ is empty. Then $Q(t,\delta) \lesssim N_2(t,\delta)$ thanks to \eqref{ineq:Q} for any $\delta \in \intervalleff{0}{t}$, which means that
	\begin{equation}\label{ineq:Qsimple}
	Q(t,\delta) \lesssim (\delta H)^\frac{1}{2} \leq t^\frac{1}{2}(1+M_{2+\varepsilon}(T))^\frac{4}{7},
	\end{equation}
	so that \eqref{ineq:Q2} is automatically verified.
	Now we suppose that there exists $\delta_*(t) \in \intervalleof{0}{t}$ such that $N_1(t,\delta_*(t)) = N_2(t,\delta_*(t))$. It comes:
	\begin{equation}\label{ineq:delta*H}
	Q(t,\delta_*(t)) =  (\delta_*(t)^{-1}H)^\frac{3}{8}.
	\end{equation}
	Then we use the inequality \eqref{ineq:Q} again so $Q(t,\delta_*(t)) \lesssim N_1(t,\delta_*(t))+N_2(t,\delta_*(t))=2N_2(t,\delta_*(t)) \lesssim (\delta H)^\frac{1}{2}$, which implies that
	\begin{equation}\label{ineq:delta*}
	H^{-\frac{1}{7}} \lesssim  \delta_*(t)
	\end{equation}
	and again using \eqref{ineq:delta*H} we obtain
	\begin{align*}
	Q(t,\delta_*(t)) \lesssim H^\frac{3}{7}.
	\end{align*}
	Now let $c_*^{-1}$ be the implicit constant in \eqref{ineq:delta*}, which depends only on the constants in \eqref{conserved}, thanks to \eqref{ineq:delta*} we can write for any $t \in \intervalleff{c_*H^{-\frac{1}{7}}}{T}$
	\begin{equation}\label{ineq:QH17}
	Q(t,c_*H^{-\frac{1}{7}}) \lesssim H^\frac{3}{7}
	\end{equation}
	Then for any such $t$, we can write $t=nc_*H^{-\frac{1}{7}}+r$ with $n\in \NN^*$ and $r < c_*H^{-\frac{1}{7}}$ and thanks to the last inequality we obtain
	\begin{align*}
	Q(t,t) & \leq Q(r,r)+\sum_{p=1}^{n} Q(pc_*H^{-\frac{1}{7}}+r,c_*H^{-\frac{1}{7}})\\
	& \lesssim (rH)^\frac{1}{2}+n H^\frac{3}{7}\\
	& \lesssim c_*(rH)^\frac{1}{2}+nc_*H^{-\frac{1}{7}} H^\frac{4}{7}\\
	& \lesssim c_* t^\frac{1}{2}H^\frac{1}{2}+t H^\frac{4}{7}
	\end{align*}
	So that finally for all $t \in \intervalleff{c_*H^{-\frac{1}{7}}}{T}$ we have
	\begin{equation}\label{ineq:Q21}
	Q(t,t) \lesssim (t^\frac{1}{2}+t)H^\frac{4}{7}.
	\end{equation}
	Lastly, if $t \leq c_*H^{-\frac{1}{7}}$ then thanks to \eqref{ineq:Qsimple} and \eqref{ineq:delta*} we can write
	\begin{equation}
	Q(t,t) \lesssim (tH)^\frac{1}{2}.
	\end{equation}
	This concludes the proof of \cref{prop:Q2} because $H>1$. 
\end{proof}
\textbf{Acknowledgments:} The author would like to thank Fr\'ed\'erique Charles, Bruno Despr\'es and Mikaela Iacobelli for all their very helpful comments related to this manuscript.
\printbibliography

\end{document}